\documentclass[a4paper,oneside]{amsart}
\usepackage{mathrsfs}
\usepackage{amsfonts}
\usepackage{amssymb}
\usepackage{amsxtra}
\usepackage{dsfont}

\usepackage[english,polish]{babel}

\newcommand{\pl}[1]{\foreignlanguage{polish}{#1}}

\usepackage{color}

\theoremstyle{plain}
\newtheorem{theorem}{Theorem}
\newtheorem{proposition}{Proposition}[section]

\newtheorem{lemma}[proposition]{Lemma}
\theoremstyle{definition}
\newtheorem{definition}{Definition}[section]
\newtheorem{remark}{Remark}[section]

\numberwithin{equation}{section}

\newcounter{thm}

\theoremstyle{plain}
\newtheorem{main_theorem}[thm]{Theorem}

\newcommand{\RR}{\mathbb{R}}
\newcommand{\ZZ}{\mathbb{Z}}
\newcommand{\TT}{\mathbb{T}}
\newcommand{\CC}{\mathbb{C}}
\newcommand{\NN}{\mathbb{N}}
\newcommand{\QQ}{\mathbb{Q}}

\newcommand{\boldA}{\mathbf{A}}

\newcommand{\calC}{\mathcal{C}}
\newcommand{\calP}{\mathcal{P}}
\newcommand{\calQ}{\mathcal{Q}}
\newcommand{\calF}{\mathcal{F}}

\newcommand{\seq}[2]{{#1}: {#2}}
\newcommand{\ind}[1]{{\mathds{1}_{{#1}}}}

\renewcommand{\atop}[2]{\substack{{#1}\\{#2}}}
\newcommand{\norm}[1]{{\left\lvert #1 \right\rvert}}
\newcommand{\sprod}[2] {{#1 \cdot #2}}

\newcommand{\abs}[1]{{\lvert {#1} \rvert}}

\newcommand{\vnorm}[1]{{\left\lVert {#1} \right\rVert}}

\title[Discrete operators of Radon types]
{Square function estimates for discrete
Radon transforms
}

\author{Mariusz Mirek}
\address{Mariusz Mirek \\
	Universit\"{a}t Bonn \\
	Mathematical Institute\\
	Endenicher Allee 60\\
	D--53115 Bonn \\
	Germany \&
	Instytut Matematyczny\\
	Uniwersytet \pl{Wroc{\lll}awski}\\
	Pl. Grun\-waldzki 2/4\\
	50-384 \pl{Wroc{\lll}aw}\\
	Poland}
\email{mirek@math.uni-bonn.de}

\begin{document}
\selectlanguage{english}

\begin{abstract}
  We show $\ell^p\big(\ZZ^d\big)$ boundedness, for $p\in(1, \infty)$, of discrete
  singular integrals of Radon type with the aid of appropriate square
  function estimates, which can be thought as a discrete counterpart
  of the Littlewood--Paley theory. It is a  very powerful approach which
  allows us to proceed as in the continuous case. 
\end{abstract}

\maketitle

\section{Introduction}

\label{sec:1}
Assume that $K \in \calC^1\big(\RR^k \setminus \{0\}\big)$ is a Calder\'{o}n--Zygmund kernel
satisfying the differential inequality
\begin{align}
	\label{eq:3}
	\norm{y}^k \abs{K(y)} + \norm{y}^{k+1} \norm{\nabla K(y)} \leq 1
\end{align}
for all $y \in \RR^k$ with $\norm{y} \geq 1$ and the cancellation condition
\begin{align}
	\label{eq:4}
	\sup_{\lambda \geq 1}
	\Big\lvert
	\int\limits_{1 \leq \norm{y} \leq \lambda} K(y) {\rm d} y
	\Big\rvert
	\leq 1.
\end{align}
Let $\calP=(\calP_1,\ldots, \calP_{d_0}): \ZZ^{k} \rightarrow \ZZ^{d_0}$ be a mapping where 
each component $\calP_j:\ZZ^{k} \rightarrow \ZZ$ is
an integer-valued polynomial of $k$ variables with $\calP_j(0) =
0$. In the present article, as in \cite{iw}, we
are interested in the discrete  singular Radon transform $T^{\calP}$
defined  by 
\begin{align}
  \label{eq:31}
T^{\calP} f(x)
=
\sum_{y\in\ZZ^k\setminus\{0\}} f\big(x - \calP(y)\big) K(y)  
\end{align}
for a finitely
supported function $f: \ZZ^{d_0} \rightarrow \RR$.
We prove the following.
\begin{main_theorem}
	\label{thm:0}
        For every $p \in (1, \infty)$ there is $C_p > 0$ such that for all
		$f \in \ell^p\big(\ZZ^{d_0}\big)$
                \begin{align}
                  \label{eq:35}
		\big\lVert
		T^\calP f 
		\big\rVert_{\ell^p}
		\leq
		C_p \vnorm{f}_{\ell^p}.
                \end{align}
		Moreover, the constant $C_p$ is independent of the coefficients of the polynomial mapping 
		$\calP$.
\end{main_theorem}
Theorem \ref{thm:0} was proven by Ionescu and Wainger in \cite{iw}. 
The operator $T^{\calP}$ is a discrete analogue of the continuous
Radon transform $R^{\calP}$ defined by
\begin{align}
  \label{eq:36}
R^{\calP}f(x)
={\rm p.v.}\int_{\RR^k} f\big(x - \calP(y)\big) K(y){\rm d}y.  
\end{align}
Nowadays the operators $R^{\calP}$ and their $L^p\big(\RR^{d_0}\big)$
boundedness properties for $p\in(1, \infty)$ are very well understood. We refer to
\cite{bigs} for a detailed exposition and the references given there,
and see also \cite{cnsw} for more general cases and more
references. The key ingredient in proving $L^p\big(\RR^{d_0}\big)$
bounds for $R^{\calP}$ is the Littlewood--Paley theory. More
precisely, we begin with $L^2\big(\RR^{d_0}\big)$ theory which, based
on some oscillatory integral estimates for dyadic pieces of the
multiplier corresponding to $R^{\calP}$, provides bounds with
acceptable decays.  Then appealing to the Littlewood--Paley theory and
interpolation it is possible to obtain general
$L^p\big(\RR^{d_0}\big)$ bounds for all $p\in(1, \infty)$.
Now, one would like to follow the same scheme in the discrete
case. However, the situation for $T^{\calP}$ is much more complicated
due to arithmetic nature of this operator. Although
$\ell^2\big(\ZZ^{d_0}\big)$ theory is based on estimates for some
oscillatory integrals or rather exponential sums associated with dyadic pieces of the multiplier
corresponding to $T^{\calP}$ as it was shown in \cite{iw}, the
$\ell^p\big(\ZZ^{d_0}\big)$ theory did not fall under the
Littlewood--Paley paradigm as it is in the continuous case.

The main aim of this paper is to give a new proof of Theorem \ref{thm:0} using
square function techniques. We construct a suitable square function
which allows us to proceed as in the continuous case to obtain
$\ell^p\big(\ZZ^{d_0}\big)$ theory for the operator \eqref{eq:31}.
Our square function gives a new insight for these sort of problems
(see especially \cite{mst1} and \cite{mst2}) and
can be thought as a discrete counterpart of the Littlewood--Paley
theory.

\subsection{Outline of the strategy of our proof} Recall from \cite[Chapter 6, \S4.5, Chapter 13, \S5.3]{bigs} that
given a kernel $K$ satisfying \eqref{eq:3} and \eqref{eq:4} there are functions 
$\big(\seq{K_n}{n \geq 0}\big)$ and a constant $C > 0$ such that for $x \neq 0$
\begin{align}
\label{eq:87}
K(x) = \sum_{n = 0}^\infty K_n(x),  
\end{align}
where for each $n \geq 0$, the kernel $K_n$ is supported inside $2^{n-2} \leq \norm{x} \leq 2^n$,
satisfies
\begin{equation}
	\label{eq:88}
	\norm{x}^k \abs{K_n(x)} + \norm{x}^{k+1} \norm{\nabla K_n(x)} \leq C
\end{equation}
for all $x \in \RR^k$ such that $\norm{x} \geq 1$, and has integral
$0$, provided that $n \in\NN$. Thus in view of \eqref{eq:88}, instead
of \eqref{eq:35}, it suffices to show that for every $p\in(1, \infty)$
there is a constant $C_p>0$ such that
\begin{align}
  \label{eq:89}
\Big\|\sum_{n\ge0}T_n^{\calP}f\Big\|_{\ell^p}\le C_p\|f\|_{\ell^p}  
\end{align}
for all $f\in\ell^p\big(\ZZ^d\big)$, where
\begin{align}
  \label{eq:90}
T_n^{\calP} f(x)
=
\sum_{y\in\ZZ^k} f\big(x - \calQ(y)\big) K_n(y)    
\end{align}
for each $n\in\ZZ$.

As we mentioned before the proof of inequality \eqref{eq:89} will
strongly follow the scheme of the proof of the corresponding
 inequality from the continuous setup. Now we describe the key points
 of our approach. To avoid some technicalities assume
that $\calP(x)=(x^d,\ldots, x)$ is a moment curve for some $d=d_0\ge2$
and $k=1$. Let $m_n$ be the Fourier multiplier associated with the
operator $T_n^{\calP}$, i.e. $T_n^{\calP}f=\mathcal F^{-1}\big(m_n\hat
f\big)$. As in \cite{mst1} and \cite{mst2} we introduce a family of
appropriate projections $\big(\Xi_n(\xi): n\ge0\big)$ which will
localize asymptotic behaviour of $m_n(\xi)$. Namely, let
$\eta$ be a smooth bump function with a small support, fix $l \in
\NN$ and define for each integer $n\ge0$ the following projections
\begin{align}
\label{eq:201}
	\Xi_{n}(\xi)
	=\sum_{a/q \in\mathscr{U}_{n^l}}
	 \eta\big(\mathcal E_n^{-1}(\xi - a/q)\big)
      \end{align}
where  $\mathcal E_n$
is a diagonal $d\times d$  matrix with positive entries
$\big(\varepsilon_{j}: 1\le j\le d\big)$ such that
$\varepsilon_{j}\le e^{-n^{1/5}}$   and
\[
\mathscr{U}_{n^l}=\big\{a/q \in \TT^d\cap\QQ^d : a=(a_1, \ldots, a_d) \in
\NN^d_q \text{ and } \mathrm{gcd}(a_1, \ldots, a_d, q)=1 \text{ and }
q\in P_{n^l}\big\}
\]
for some family $P_{n^l}$ such that $\NN_{n^l}\subseteq
P_{n^l}\subseteq \NN_{e^{n^{1/10}}}$. All details are described in
Section \ref{sec:6}.  Exploiting the ideas of Ionescu and Wainger from
\cite{iw}, we prove that for every $p\in(1, \infty)$ there is a constant $C_{l,
  p}>0$ such that
\begin{align}
  \label{eq:202}
\big\|\calF^{-1}\big(\Xi_{n}\hat{f}\big)\big\|_{\ell^p}\le C_{l,
  p}\log (n+2)\|f\|_{\ell^p}.  
\end{align}
Inequality \eqref{eq:202} will be essential in our proof. Observe that
\eqref{eq:201} allows us to dominate \eqref{eq:89} as follows
\begin{align}
  \label{eq:91}
\Big\|\sum_{n\ge0}T_n^{\calP}f\Big\|_{\ell^p}\le\Big\|\sum_{n\ge0}\mathcal
F^{-1}\big(m_n\Xi_n\hat f\big)\Big\|_{\ell^p}
+\Big\|\sum_{n\ge0}\mathcal
F^{-1}\big(m_n(1-\Xi_n)\hat f\big)\Big\|_{\ell^p}  
\end{align}
and we can employ the ideas from  the circle method of Hardy and
Littlewood, which are implicit in the behaviour  of the projections
$\Xi_n$ and $1-\Xi_n$.
Namely, the second norm on the right-hand side of \eqref{eq:91} is bounded,
since the multiplier $m_n(1-\Xi_n)$ is highly oscillatory. Thus
appealing to \eqref{eq:202} and a variant of Weyl's inequality
with logarithmic decay which has been proven in \cite{mst1}, see Theorem
\ref{thm:3}, we can conclude that there is a constant $C_p>0$ such
that for each $n\ge0$ we have
\[
\big\|\calF^{-1}\big(m_{n}(1-\Xi_{n})\hat{f}\big)\big\|_{\ell^p}\le C_{
  p}(n+1)^{-2}\|f\|_{\ell^p}.
\]  
Now 
the whole difficulty  lies in proving 
\begin{align}
  \label{eq:32}
  \Big\|\sum_{n\ge0}\mathcal
F^{-1}\big(m_n\Xi_n\hat f\big)\Big\|_{\ell^p}\le C_p\|f\|_{\ell^p}.
\end{align}
For this purpose we  construct new multipliers of the form 
\begin{align}
\label{eq:47}
  \Delta_{n, s}^j(\xi)=\sum_{a/q\in\mathscr
    U_{(s+1)^l}\setminus\mathscr
    U_{s^l}}\big(\eta\big(\mathcal E_{n+j}(\xi-a/q)\big)-\eta\big(\mathcal E_{n+j+1}(\xi-a/q)\big)\big)\eta\big(\mathcal E_{s}(\xi-a/q)\big)
\end{align} 
such that
\[
\Xi_n(\xi)\simeq\sum_{j\in\ZZ}\sum_{s\ge0}\Delta_{n, s}^j(\xi).
\]
Moreover, we will be able to show in Theorem \ref{thm:30}, using Theorem \ref{th:3}, that for
each $p\in(1, \infty)$ there is a constant $C_p>0$ such that
\begin{align}
  \label{eq:48}
   \bigg\|\Big(\sum_{n\in\ZZ}\big|\mathcal
  F^{-1}\big(\Delta_{n,
    s}^j\hat{f}\big)\big|^2\Big)^{1/2}\bigg\|_{\ell^p}
	\le 
	C_p \log(s+2) \|f\|_{\ell^p}.
 \end{align}
 for any $s\ge0$, uniformly in $j\in\ZZ$. Estimate \eqref{eq:48} can
 be thought as a discrete counterpart of the Littlewood--Paley inequality,
 see Theorem \ref{thm:30}. This is the key ingredient in our proof
 which combined with the robust $\ell^2\big(\ZZ^d\big)$ estimate
\begin{align}
  \label{eq:49}
   \bigg\|\Big(\sum_{n\in\ZZ}\big|\mathcal
  F^{-1}\big(m_n\Delta_{n,
    s}^j\hat{f}\big)\big|^2\Big)^{1/2}\bigg\|_{\ell^2}
	\le 
	C2^{-\varepsilon|j|}(s+1)^{-\delta l} \|f\|_{\ell^2}
\end{align}
allows us to deduce \eqref{eq:32}. The last bound follows, since for
each $a/q\in\mathscr{U}_{(s+1)^l}\setminus \mathscr{U}_{s^l}$ we have
\[
m_{n}(\xi)=G(a/q)\Phi_{n}(\xi-a/q)+\mathcal O\big(2^{-n/2}\big)
\]
where $G(a/q)$ is the Gaussian sum and $\Phi_{n}$ is an integral
counterpart of $m_{n}$, precise definitions can be found at beginning
of Section \ref{sec:4}. This observation  leads to
\eqref{eq:49}, because $|G(a/q)|\le Cq^{-\delta}$ and
$q\ge s^l$ if $a/q\in\mathscr{U}_{(s+1)^l}\setminus
\mathscr{U}_{s^l}$. The decay in $|j|$ in \eqref{eq:49} follows from
the assumption on the support of $\Xi_{n, s}^j$ and the behaviour of
$\Phi_n(\xi-a/q)$, see Section \ref{sec:4} for more details. 

The ideas of exploiting projections \eqref{eq:201} have been initiated
in \cite{mst1} in the context of  $\ell^p\big(\ZZ^{d_0}\big)$
boundedness of maximal functions corresponding respectively to the averaging
Radon operators
\begin{align}
\label{eq:39}
	 M_N^{\calP} f(x)
	=N^{-k} \sum_{y\in\NN_N^k}
	f\big(x-\calP(y)\big)
\end{align}
where $\NN^k_N = \{1, 2, \ldots, N\}^k$, and the truncated singular
Radon transforms
\begin{align}
  \label{eq:51}
T_N^{\calP} f(x)
=
\sum_{y\in\mathbb B_N\setminus\{0\}} f\big(x - \calP(y)\big) K(y)  
\end{align}
where $\mathbb B_N=\{x\in\ZZ^k: |x|\le N\}$. These ideas, on the one
hand, resulted in a new proof for Bourgain's maximal operators
\cite{bou1, bou2, bou}. On the other hand, turned out to be flexible
enough and attack $\ell^p\big(\ZZ^{d_0}\big)$ boundedness of maximal
functions for operators with signs like in \eqref{eq:51}. In fact, in
\cite{mst1} we provided some vector-valued estimates for the maximal
functions associated with \eqref{eq:39} and \eqref{eq:51}. These
estimates found applications in variational estimates for
\eqref{eq:39} and \eqref{eq:51}, which have been the subject of
\cite{mst2}.  Our approach falls within the scope of a general scheme
which has been recently developed in \cite{mst1} and \cite{mst2} and
resulted in some unification in the theory of discrete analogues in
harmonic analysis. The novelty of this paper is that it provides some
counterpart of the Littlewood--Paley square function which is useful
in the problems with arithmetic flavor. Furthermore, this square
function theory  is also an invaluable ingredient in the
estimates of variational seminorm in \cite{mst2}.

The paper is organized as follows.  In Section \ref{sec:6} we prove
Theorem \ref{th:3} which is essential in our approach, and guarantees
\eqref{eq:202}.  Ionescu and Wainger in \cite{iw} proved this result
with $(\log N)^D$ loss in norm where $D>0$ is a large power.  In \cite{mst1}
we provided a slightly different proof and showed that $\log N$ is
possible. Moreover $\log N$ loss is sharp for the method which we
used.  Since Theorem \ref{th:3} is a deep theorem, which uses the most
sophisticated tools developed to date in the field of discrete
analogues, we have decided, for the sake of
completeness, to provide necessary details.  In Section \ref{sec:4} we
prove Theorem \ref{thm:0}. To understand more quickly the proof of
Theorem \ref{thm:0}, the reader may begin by looking at Section
\ref{sec:4} first. These sections can be read independently, assuming
the results from Section \ref{sec:6}.

\subsection{Basic reductions}
We set 
\[
N_0 = \max\{ \deg \calP_j : 1 \leq j \leq d_0\}
\]
and  define
the set
\[
\Gamma =
\big\{
	\gamma \in \ZZ^k \setminus\{0\} : 0 \leq \gamma_j \leq N_0
	\text{ for each } j = 1, \ldots, k
\big\}
\]
with the lexicographic order. Let $d$ be the cardinality of $\Gamma$. Then we can identify
$\RR^d$ with the space of all vectors whose coordinates are labeled by
multi-indices $\gamma \in \Gamma$.
Next we
introduce the canonical polynomial mapping
$$
\calQ = \big(\seq{\calQ_\gamma}{\gamma \in \Gamma}\big) : \ZZ^k \rightarrow \ZZ^d
$$
where $\calQ_\gamma(x) = x^\gamma$ and $x^\gamma=x_1^{\gamma_1}\cdot\ldots\cdot x_k^{\gamma_k}$.
The canonical polynomial mapping $\calQ$ determines anisotropic
dilations. Namely, let $A$ be 
 a diagonal $d \times d$
matrix such that
$$
(A v)_\gamma = \abs{\gamma} v_\gamma
$$
for any $v\in\RR^d$ and $\gamma\in\Gamma$, where $|\gamma|=\gamma_1+\ldots+\gamma_k$. Then for every $t > 0$ we set
$$
t^{A}=\exp(A\log t)
$$
i.e. $t^A x=(t^{|\gamma|}x_{\gamma}: \gamma\in \Gamma)$ for
$x\in\RR^d$ and we see that  $\calQ(tx)=t^A\calQ(x)$.

 Observe also that each $\calP_j$ can be expressed as
$$
\calP_j(x) = \sum_{\gamma \in \Gamma} c_j^\gamma x^\gamma
$$
for some $c_j^\gamma \in \RR$. Moreover, the coefficients
$\big(\seq{c_j^\gamma}{\gamma \in \Gamma, j \in \{1, \ldots,
  d_0\}}\big)$ define a linear transformation $L: \RR^d \rightarrow
\RR^{d_0}$ such that $L\calQ = \calP$. Indeed, it is enough to set
$$
(L v)_j = \sum_{\gamma \in \Gamma} c_j^\gamma v_\gamma
$$
for each $j \in \{1, \ldots, d_0\}$ and $v \in \RR^d$. Now instead of
proving Theorem \ref{thm:0} we show the following. 
\begin{main_theorem}
\label{thm:7}
        For every $p \in (1, \infty)$ there is $C_p > 0$ such that for all
		$f \in \ell^p\big(\ZZ^{d}\big)$
                \begin{align}
\label{eq:37}
		\big\lVert
		T^\calQ f 
		\big\rVert_{\ell^p}
		\leq
		C_p \vnorm{f}_{\ell^p}.
                \end{align}
\end{main_theorem}
In view of \cite[Section 11]{bigs} we can perform some
lifting procedure which allows us to replace the underlying polynomial
mapping $\calP$ from \eqref{eq:35} by the canonical polynomial mapping
$\calQ$. Moreover, it shows  that
\eqref{eq:37} implies \eqref{eq:35} with the same constant
$C_p$, see also \cite{iw} for more details. Therefore, the matters are reduced to proving \eqref{eq:37} for
the canonical polynomial mapping.  The advantage of working with the
canonical polynomial mapping $\calQ$ is that it has all coefficients
equal to $1$, and the uniform bound in this case is immediate.  From
now on for simplicity of notation we will write $T=T^{\calQ}$.

\subsection{Notation}
Throughout the whole article $C > 0$ will stand for a positive constant (possibly
large constant) whose value may change from occurrence to
occurrence.  If there is an absolute constant $C>0$ such that $A\le CB$ ($A\ge
CB$) then we will  write $A \lesssim B$ ($A \gtrsim
B$).  Moreover,  we will write $A \simeq B$ if $A \lesssim B$ and $A\gtrsim
B$ hold simultaneously, and  we will denote $A \lesssim_{\delta} B$
($A \gtrsim_{\delta} B$) to indicate that the constant $C>0$ depends
on some $\delta > 0$. Let $\NN_0 = \NN \cup \{0\}$ and for $N \in \NN$ we
set
\[
    \NN_N = \big\{1, 2, \ldots, N\big\}.
\]
For a vector $x \in \RR^d$ we will use the following norms
\[
    \norm{x}_\infty = \max\{\abs{x_j} : 1 \leq j \leq d\}, \quad \text{and} \quad
    \norm{x} = \Big(\sum_{j = 1}^d \abs{x_j}^2\Big)^{1/2}.
\]
If $\gamma$ is a multi-index from $\NN_0^k$ then $\norm{\gamma} = \gamma_1 + \ldots
+ \gamma_k$. Although, we use $|\cdot|$ for the length of a
multi-index $\gamma\in \NN_0^k$
and the Euclidean norm of $x\in\RR^d$, their meaning will be always
clear from the context and it will cause no confusions in
the sequel.

\subsection*{Acknowledgments}
I am grateful to Wojtek Samotij for explaining  a very
nice probabilistic argument which simplifies the proof of Lemma \ref{lem:12}.

\section{ Ionescu--Wainger type multipliers}
\label{sec:6}

For a function 
$f \in L^1\big(\RR^d\big)$ let $\calF$ denote the Fourier transform on $\RR^d$ defined as
$$
\calF f(\xi) = \int_{\RR^d} f(x) e^{2\pi i \sprod{\xi}{x}} {\rm d}x.
$$
If $f \in \ell^1\big(\ZZ^d\big)$ we set
$$
\hat{f}(\xi) = \sum_{x \in \ZZ^d} f(x) e^{2\pi i \sprod{\xi}{x}}.
$$
To simplify the notation we denote by $\mathcal F^{-1}$ the inverse
Fourier transform on $\RR^d$ and the inverse Fourier transform on the
torus $\TT^d\equiv[0, 1)^d$ (Fourier coefficients). The meaning of
$\calF^{-1}$ will be always clear from the context.  Let $\eta: \RR^d
\rightarrow \RR$ be a smooth function such that $0 \leq \eta(x) \leq
1$ and
\[
\eta(x) =
\begin{cases}
	1 & \text{ for } \norm{x} \leq 1/(16 d),\\
	0 & \text{ for } \norm{x} \geq 1/(8 d).
\end{cases}
\]
\begin{remark}
\label{rem:1}
We will additionally assume that $\eta$ is a convolution of two non-negative smooth functions
$\phi$ and $\psi$ with compact supports contained in $(-1/(8d), 1/(8d))^d$.   
\end{remark}

This section is intended to prove Theorem \ref{th:3}, which
is inspired by  the ideas of Ionescu and Wainger from
\cite{iw}.  Let $\rho>0$ and  for every $N\in\NN$ define
\[
N_0=\lfloor
N^{\rho/2}\rfloor+1 \qquad \text{and} \qquad Q_0=(N_0!)^D
\]
where
$D=D_{\rho}=\lfloor 2/\rho\rfloor+1$. Let  $\mathbb P_N=\mathbb P\cap
(N_0, N]$ where $\mathbb P$ is the set
of all prime numbers. For any
$V\subseteq\mathbb P_N$ we define
\[
\Pi(V)=\bigcup_{k\in\NN_D}\Pi_k(V)
\]
where for any $k\in\NN_D$
\[
\Pi_k(V)=\big\{p_{1}^{\gamma_{1}}\cdot\ldots\cdot
p_{k}^{\gamma_{k}}: \ \gamma_{l}\in\NN_D\  \text{and}\ p_{l}\in
V\ \text{are distinct for all $1\le l\le k$} \big\}.
\]
In other words $\Pi(V)$ is the set of all products of primes factors from $V$ of
length at most $D$, at powers between $1$ and $D$. Now we introduce
the sets
\[
 P_N=\big\{q=Q\cdot w: Q|Q_0\ \text{and}\ w\in
\Pi(\mathbb P_{N})\cup\{1\}\big\}.
\]
It is not difficult
to see  that every integer $q\in\NN_N$ can be uniquely
written as $q=Q\cdot w$ where $Q|Q_0$ and $w\in \Pi(\mathbb
P_N)\cup\{1\}$.  Moreover, for sufficiently large $N\in\NN$ we have 
\[q=Q\cdot w\le Q_0\cdot w\le (N_0!)^DN^{D^2}\le
e^{N^{\rho}}
\]
thus  
we have 
$\NN_N\subseteq P_N\subseteq\NN_{e^{N^{\rho}}}$. Furthermore, if
$N_1\le N_2$ then  $P_{N_1}\subseteq P_{N_2}$.

For a subset $S\subseteq\NN$ we define 
\[
\mathcal R(S)=\big\{a/q \in \TT^d\cap\QQ^d : a \in A_q \text{ and }
q\in S\big\}
\] 
where for each $q\in\NN$ 
\[
A_q=\big\{a\in\NN_q^d: \mathrm{gcd}\big(q, \mathrm(a_{\gamma}: \gamma\in\Gamma)\big)\big\}.
\]
Finally, for each $N\in\NN$ we will consider the sets
\begin{align}
  \label{eq:156}
\mathscr{U}_N=\mathcal R(P_N).
\end{align}
It is easy to see, if $N_1\le N_2$ then  $\mathscr{U}_{N_1}\subseteq \mathscr{U}_{N_2}$.

 We will assume that $\Theta$ is a multiplier on $\RR^d$ and
for every $p\in(1, \infty)$ there is a constant $\boldA_p >0 $ such that for every
$f\in L^2\big(\RR^d\big)\cap L^p\big(\RR^d\big)$ we have
\begin{align}
  \label{eq:155}
  \big\lVert\calF^{-1}\big(\Theta\mathcal F f\big)\big\rVert_{L^p}
  \leq
  \boldA_p \vnorm{f}_{L^p}.
\end{align}

For each $N \in \NN$
we define new periodic  multipliers
\begin{align}
	\label{eq:74}
	\Delta_N(\xi)
	=\sum_{a/q \in\mathscr{U}_N}
	 \Theta(\xi - a/q) \eta_N(\xi - a/q)
      \end{align}
where $\eta_N(\xi)=\eta\big(\mathcal E_N^{-1}\xi\big)$ and $\mathcal E_N$
is a diagonal $d\times d$  matrix with  positive entries
$(\varepsilon_{\gamma}: \gamma\in\Gamma)$ such that
$\varepsilon_{\gamma}
\le e^{-N^{2\rho}}$.
The main result is the following.
\begin{theorem}
	\label{th:3}
	Let $\Theta$ be a  multiplier on $\RR^d$ obeying \eqref{eq:155}. 
	Then for every $\rho>0$ and $p\in(1, \infty)$ there is a
        constant $C_{\rho, p} > 0$ such that
	for any $N\in\NN$ and $f \in \ell^p\big(\ZZ^d\big)$
	\begin{align}
		\label{eq:12}
		\big\lVert
		\calF^{-1}\big(\Delta_{N}
		\hat{f}\big)\big\rVert_{\ell^p} 
		\leq C_{\rho, p} (\boldA_p+1)\log N
		\vnorm{f}_{\ell^p}.
	\end{align}
\end{theorem}
The main constructing blocks have been gathered in the next three
subsections. Theorem \ref{th:3} is a consequence of Theorem
\ref{thm:5} and Proposition \ref{prop:5} proved below.  To prove
Theorem \ref{th:3} we find some $C_{\rho}>0$ and disjoint sets
$\mathscr{U}_N^i\subseteq\mathscr{U}_N$ such that
\[
\mathscr{U}_N=\bigcup_{1\le i\le C_{\rho}\log N}\mathscr{U}_N^i
\]
and we show that $\Delta_N$ with the
summation restricted to  $\mathscr{U}_N^i$ is bounded on
$\ell^p\big(\ZZ^d\big)$ for every $p\in(1, \infty)$. In order to
construct $\mathscr{U}_N^i$ we need a suitable partition of
integers from the set $\Pi(\mathbb P_{N})\cup\{1\}$, see also
\cite{iw}.

\subsection{Fundamental combinatorial lemma} 
We begin with
\begin{definition}
\label{def:1}
A subset $\Lambda\subseteq \Pi(V)$ has
$\mathcal O$ property if there is $k\in\NN_D$ and there are sets 
$S_1, S_2,\ldots, S_k$ with the following properties:
\begin{enumerate}
\item[(i)] for each $1\le j\le k$ there is $
  \beta_j\in\NN$ such that $S_j=\{q_{j, 1},\ldots, q_{j, \beta_j}\}$;
\item[(ii)] for every $q_{j, s}\in S_j$ there are $p_{j, s}\in V$ and
  $\gamma_j\in\NN_D$ such that $q_{j, s}=p_{j, s}^{\gamma_j}$;
\item[(iii)] for every $w\in W$ there are unique numbers $q_{1, s_1}\in
  S_1,\ldots, q_{k, s_k}\in S_k$ such that $w=q_{1,
    s_1}\cdot\ldots\cdot q_{k, s_k}$;
\item[(iv)] if $(j, s)\not=(j', s')$ then $(q_{j, s}, q_{j', s'})=1$. 
\end{enumerate}  
\end{definition}
Now three comments are in order. 
\begin{itemize}
\item[(i)] The set $\Lambda=\{1\}$ has $\mathcal O$ property
corresponding to $k=0$.
\item[(ii)] If $\Lambda$ has $\mathcal O$ property,
then each subset $\Lambda'\subseteq \Lambda$ has $\mathcal O$ property as
well. 
\item[(iii)] 
  If a set $\Lambda$ has
$\mathcal O$ property then each element of $\Lambda$ has the same number of
prime factors $k\le D$.
\end{itemize}
The main result is the following.
\begin{lemma}
\label{lem:12}
For every $\rho>0$ there exists a constant $C_{\rho}>0$ such that for
every $N\in\NN$ the set $\mathscr{U}_N$ can be written as a
disjoint union of at most $C_{\rho}\log N$ sets
$\mathscr{U}_N^i=\mathcal R(P_N^i)$ where
\begin{align}
  \label{eq:14}
   P_N^i =\big\{
  q=Q\cdot w: Q|Q_0\ \text{and}\ w\in {\Lambda}_i(\mathbb P_N)\big\}
\end{align}
and 
${\Lambda}_i(\mathbb P_N)
\subseteq \Pi(\mathbb P_N)\cup\{1\}$ has $\mathcal O$ property
for each integer $1\le i\le C_{\rho}\log N$. 
\end{lemma}

\begin{proof}
  We have to prove that for every $V\subseteq \mathbb P_N$ the set
  $\Pi(V)$ can be written as a disjoint union of at most $C_k\log N$
  sets with $\mathcal O$ property. Fix $k\in\NN_D$, let 
  $\gamma=(\gamma_1,\ldots,\gamma_k)\in\NN_D^k$ be a multi-index
  and observe that 
\[
\Pi_k(V)=\bigcup_{\gamma\in\NN_D^k}\Pi_k^{\gamma}(V)
\]
where 
\[
\Pi_k^{\gamma}(V)=\big\{p_{1}^{\gamma_1}\cdot\ldots\cdot
p_{k}^{\gamma_k}:\ \text{$p_l\in V$ are distinct for all $1\le l\le k$} \big\}.
\]
Since there are $D^k$ possible choices of exponents
$\gamma_1,\ldots,\gamma_k\in\NN_D$ when $k\in\NN_D$, it only suffices
to prove that every $\Pi_k^{\gamma}(V)$ can be
partitioned into a union (not necessarily disjoint) of at most
$C_k\log N$ sets with  $\mathcal O$ property.

We claim that for each $k\in\NN$ there is a constant $C_k>0$ and a family 
\begin{align}
  \label{eq:11}
\pi=\big\{\pi_i(V): 1\le i\le C_k\log |V|\big\}  
\end{align}
 of partitions of $V$ such that 
 \begin{itemize}
 \item[(i)] 
for every $1\le i\le C_k\log |V|$ 
each $\pi_i(V)=\{V_1^i,\ldots, V_k^i\}$ consists of pairwise disjoint
subsets of $V$
and
$V=V_1^i\cup\ldots\cup V_k^i$; 
\item[(ii)] for every $E\subseteq V$ with at least $k$
elements there exists $\pi_i(V)=\{V_1^i,\ldots, V_k^i\}\in\pi$ such that $E\cap
V_j^i\not=\emptyset$ for every $1\le j\le k$.  
 \end{itemize}
Assume for a moment we have constructed a family $\pi$ as in \eqref{eq:11}. Then one
sees that for a fixed $\gamma\in\NN_D^k$ we have
\begin{align}
  \label{eq:28}
\Pi_k^{\gamma}(V)=\bigcup_{1\le i\le C_{k}\log |V|}\Pi_{k, i}^{\gamma}(V)  
\end{align}
where
\[
\Pi_{k, i}^{\gamma}(V)=\big\{p_{1}^{\gamma_1}\cdot\ldots\cdot
p_{k}^{\gamma_k}:  p_{j}^{\gamma_j}\in \big(V\cap V_j^i)^{\gamma_j}
\text{  and $V_j^i\in\pi_i(V)$ for each $1\le j\le k$}\big\}
\]
and $\big(V\cap V_j^i)^{\gamma_j}=\{p^{\gamma_j}: p\in V\cap V_j^i\}
$. Indeed, the sum on the right-hand side of \eqref{eq:28} is
contained in $\Pi_k^{\gamma}(V)$ since each $\Pi_{k, i}^{\gamma}(V)$
is. For the opposite inclusion take $p_{1}^{\gamma_1}\cdot\ldots\cdot
p_{k}^{\gamma_k}\in\Pi_k^{\gamma}(V)$ and let $E=\{p_1,\ldots, p_k\}$,
then property (ii) for the family \eqref{eq:11} ensures that there is
$\pi_i(V)=\{V_1^i,\ldots, V_k^i\}\in\pi$ such that $E\cap
V_j^i\not=\emptyset$ for every $1\le j\le k$. Therefore,
$p_{1}^{\gamma_1}\cdot\ldots\cdot p_{k}^{\gamma_k}\in\Pi_{k,
  i}^{\gamma}(V)$. Furthermore, we see that for each $1\le i\le C_k\log
N$ the sets $\Pi_{k, i}^{\gamma}(V)$ have $\mathcal O$ property.  

The proof will be completed if we construct the family $\pi$ as in
\eqref{eq:11} for the set $V$. We assume, for simplicity, that
$V=\NN_N$  but the result is
true for all $V\subseteq \NN_N$ containing at least $k$ elements. Now
it will be more comfortable to work with surjective mappings
$f:\NN_N\mapsto\NN_k$ rather than with partitions of $\NN_N$ into $k$
non-empty subsets.  It will cause no changes to us, since  every surjection 
$f:\NN_N\mapsto\NN_k$ determines a partition $\{f^{-1}[\{m\}]: 1\le
m\le k\}$ of $\NN_N$ into $k$ non-empty subsets.  

For the proof we employ a probabilistic argument.  Indeed, let
$f:\NN_N\mapsto \NN_k$ be a random surjective mapping. Assume that for
every $n\in\NN_N$ and $m\in\NN_k$ we have $\mathds{P}(\{f(n)=m\})=1/k$
independently of all other $n\in\NN_N$.  For every $E\subseteq\NN_N$
with $k$ elements we have $\mathds P(\{|f[E]|=k\})=k!/k^k$. 
It suffices to show that for some $r\simeq_k\log N$ and $f_1, \ldots,
f_r$ random surjections  we have
\[
\mathds P(\{\forall_{E\subseteq\NN_N}\ |E|=k\ \exists _{1\le l\le r}\
|f_l[E]|=k\})>0.
\]
In other words,  for each $E\subseteq\NN_N$ with cardinality $k$ it is always possible to
find, with a positive probability,  among at most $C_k\log N$ random surjections 
 at least one $f:\NN_N\mapsto \NN_k$ such that $|f[E]|=k$. Then the
 set $\{f^{-1}[\{m\}]: 1\le
m\le k\}$ is a partition of $\NN_N$ and $E\cap
f^{-1}[\{m\}]\not=\emptyset$ for every $1\le m\le k$.

 The task now is
the determine the exact value of $r\simeq_k\log N$.
Take now
$1\le r\le N$ independent random surjections $f_1, \ldots, f_r$ and observe that
\begin{multline*}
\mathds P(\{\exists_{E\subseteq\NN_N}\ |E|=k\ \forall _{1\le l\le r}\
|f_l[E]|<k\})\le \sum_{E\subseteq \NN_N:\: |E|=k}
\mathds P(\{\forall _{1\le l\le r}\
|f_l[E]|<k\})  \\
=\sum_{E\subseteq \NN_N:\:
  |E|=k}\bigg(1-\frac{k!}{k^k}\bigg)^r=\binom{N}{k}\bigg(1-\frac{k!}{k^k}\bigg)^r
\le\bigg(\frac{eN}{k}\bigg)^ke^{-r\frac{k!}{k^k}}
=e^{k\log(\frac{eN}{k})-r\frac{k!}{k^k}}.
\end{multline*}
Therefore 
\[
\mathds P(\{\exists_{E\subseteq\NN_N}\ |E|=k\ \forall _{1\le l\le r}\
|f_l[E]|<k\})<1
\]
 if and only if 
\[
r>\frac{k^{k+1}}{k!}\log\bigg(\frac{eN}{k}\bigg).
\]
Thus taking 
\[
r=\bigg\lceil\frac{k^{k+1}}{k!}\log\bigg(\frac{eN}{k}\bigg)\bigg\rceil+1\simeq
C_k\log N
\]
we see that it does the job. 
This completes the proof of Lemma \ref{lem:12}.
\end{proof}

\subsection{Further reductions and square function estimates}
Now we can write 
\begin{align*}
\Delta_N=\sum_{1\le i\le C_{\rho}\log N}\Delta_N^i
\end{align*}
where for each $1\le i\le C_{\rho}\log N$ 
\begin{align}
  \label{eq:290}
	\Delta_N^i(\xi)
	=\sum_{a/q \in\mathscr{U}_N^i}
	 \Theta(\xi - a/q) \eta_N(\xi - a/q)  
\end{align}
with $\mathscr{U}_N^i$  as in Lemma \ref{lem:12}.
The proof of Theorem \ref{th:3} will be completed if we
show that for every $p\in(1, \infty)$ and $\rho>0$, there is a
constant $C>0$ such that for any $N\in\NN$ and
$1\le i\le C_{\rho} \log N$ we have
 \begin{align}
\label{eq:17}
        \big\lVert
        \calF^{-1}\big(\Delta_N^{i} \hat{f}\big)\big\rVert_{\ell^{p}} 
		\leq C (\boldA_p+1) \vnorm{f}_{\ell^{p}}
        \end{align}
        for every $f\in\ell^p\big(\ZZ^d\big)$.  

Let 
        \begin{align}
          \label{eq:300}
\Lambda \subseteq \Pi(\mathbb P_N)\cup\{1\}          
        \end{align}
        be a set with $\mathcal O$ property, see Definition
        \ref{def:1}. Define
\[
\mathscr{U}_N^\Lambda=\mathcal R\big(\big\{q=Q\cdot w: Q|Q_0\ \text{and}\ w\in
\Lambda\big\}\big)
\]
and $\mathscr{W}_N =\mathcal R(\Lambda)$, and we introduce
\begin{align}
  \label{eq:301}
  \Delta_N^\Lambda(\xi)
	=\sum_{a/q \in\mathscr{U}_N^\Lambda}
	 \Theta(\xi - a/q) \eta_N(\xi - a/q).  
\end{align}
We show that for every $p\in(1, \infty)$ and $\rho>0$, there is a
constant $C>0$ such that for any  $N\ge 8^{\max\{p, p'\}/\rho}$ and for
any set $\Lambda$ as in \eqref{eq:300} and
        for every $f\in\ell^p\big(\ZZ^d\big)$ we have
 \begin{align}
\label{eq:302}
        \big\lVert
        \calF^{-1}\big(\Delta_N^{\Lambda} \hat{f}\big)\big\rVert_{\ell^{p}} 
		\leq C (\boldA_p+1) \vnorm{f}_{\ell^{p}}.
        \end{align}
For  $N \le 8^{\max\{p, p'\}/\rho}$ the bound in \eqref{eq:302} is obvious, since we allow the
constant $C>0$ to depend on $p$ and $\rho$. Moreover, by the
duality and interpolation, it suffices to prove \eqref{eq:302}
for $p=2r$ where $r\in\NN$. If $\Lambda=\Lambda_i(P_N)$ as in Lemma
\ref{lem:12} for some $1\le i\le
C_{\rho}\log N$, then we see that
$\mathscr{U}_N^\Lambda=\mathscr{U}_N^i$ and
$\Delta_N^{\Lambda}=\Delta_N^i$, and consequently \eqref{eq:302} implies
\eqref{eq:17} as desired. 

The function $\Theta(\xi)\eta_N(\xi)$ is regarded as a periodic function on $\TT^d$, thus
\begin{align*}
  \Delta_N^{\Lambda}(\xi)&=\sum_{a/q\in \mathscr{U}_N^\Lambda}\Theta(\xi-a/q)\eta_N(\xi-a/q)\\
  &= \sum_{b\in\NN_{Q_0}^d}\sum_{a/w\in
    \mathscr{W}_N}\Theta(\xi-b/Q_0-a/w)\eta_N(\xi-b/Q_0-a/w)
\end{align*}
where we have used the fact, that if $(q_1, q_2)=1$ then for every
$a\in\ZZ^d$, there are unique $a_1, a_2\in\ZZ^d$, such that $a_1/q_1, a_2/q_2\in[0, 1)^d$ and
\begin{align}
  \label{eq:343}
\frac{a}{q_1q_2}=\frac{a_1}{q_1}+\frac{a_2}{q_2} \pmod{\ZZ^d}.  
\end{align}
Since $\Lambda$ has $\mathcal O$ property then according to
Definition \ref{def:1} there is an integer $1\le k\le 2/\rho+1$ and
there are sets $S_1, \ldots, S_k$ such that for any $j\in\NN_k$ we
have $S_j=\{q_{j, 1},\ldots,q_{j, \beta_j} \}$ for some $\beta_j\in\NN$. 

 Now for each $j\in\NN_k$ we introduce 
\[
\mathcal U_{\{j\}}=\big\{a_{j, s}/q_{j, s}\in\TT^d\cap \QQ^d: s\in\NN_{\beta_j}\ \text{and}\ a_{j, s}\in A_{q_{j, s}}\big\}
\] 
and for any $M=\{j_1, \ldots, j_m\}\subseteq \NN_k$ let
\[
\mathcal U_M=\big\{u_{j_1}+\ldots+u_{j_m}\in\TT^d\cap \QQ^d:
u_{j_l}\in \mathcal U_{\{j_l\}}\ \text{for any}\ l\in\NN_m\big\}.
\]
For any sequence $\sigma=(s_{j_1}, \ldots, s_{j_m})\in
\NN_{\beta_{j_1}}\times\ldots\times\NN_{\beta_{j_m}}$ determined by
the set $M$, let us define 
\[
\mathcal V_M^{\sigma}=\big\{a_{j_1, s_{j_1}}/q_{j_1, s_{j_1}}+\ldots+
a_{j_m, s_{j_m}}/q_{j_m, s_{j_m}}\in\TT^d\cap \QQ^d: a_{j_l, s_{j_l}}\in A_{q_{j_l, s_{j_l}}}\ \text{for any}\ l\in\NN_m\big\}.
\]
Note that $\mathcal V_M^{\sigma}$ is a subset of $\mathcal U_M$ with
fixed denominators $q_{j_1, s_{j_1}},\ldots,q_{j_m, s_{j_m}}$. If
$M=\emptyset$ then we have $\mathcal U_M=\mathcal V_M=\{0\}$. Let
\[
\chi(\xi)=\ind{\Lambda}(\xi) \qquad \text{and} \qquad \Omega_N(\xi)=\Theta(\xi)\eta_N(\xi).
\]
Then again by \eqref{eq:343} we obtain
\begin{multline}
\label{eq:13}
  \Delta_N^\Lambda(\xi)=
\sum_{a/w\in \mathscr{W}_N}\sum_{b\in\NN_{Q_0}^d}\Theta(\xi-b/Q_0-a/w)\eta_N(\xi-b/Q_0-a/w)\\
=\sum_{s_1\in\NN_{\beta_1}}\sum_{a_{1, s_1}\in A_{q_{1, s_1}}}\ldots\sum_{s_k\in\NN_{\beta_k}}
\sum_{a_{k, s_k}\in A_{q_{k, s_k}}}
m_{a_{1, s_1}/q_{1, s_1}+\ldots+a_{k, s_k}/q_{k, s_k} }(\xi)=\sum_{u\in \mathcal U_{\NN_k}}m_{u}(\xi)
\end{multline}
where 
\begin{align}
\label{eq:27}
  m_{u}(\xi)=m_{a_{1, s_1}/q_{1, s_1}+\ldots+a_{k, s_k}/q_{k, s_k} }(\xi)=\chi(q_{1, s_1}\cdot\ldots\cdot q_{k,
    s_k})\sum_{b\in\NN_{Q_0}^d}
  \Omega_N\Big(\xi-b/Q_0-\sum_{j=1}^ka_{j, s_j}/q_{j, s_j}\Big)
\end{align}
 for $u= a_{1, s_1}/q_{1, s_1}+\ldots+a_{k, s_k}/q_{k, s_k}$.

From now on we will write, for every $u\in \mathcal
U_{\NN_k}$, that
\begin{align}
  \label{eq:25}
f_u(x)=\mathcal F^{-1}\big(m_u\hat{f}\big)(x)
\end{align}
with $f\in\ell^{2r}\big(\ZZ^d\big)$ and $r\in\NN$.
Therefore, 
\begin{align}
\label{eq:135}
   \mathcal F^{-1}\big(\Delta_N^{\Lambda}\hat{f}\big)(x)=\sum_{u\in \mathcal U_{\NN_k}}f_u(x)
\end{align}
and 
the proof of inequality  \eqref{eq:302} will follow from the
following.
\begin{theorem}
  \label{thm:500}
Suppose that $\rho>0$ and $r\in\NN$ are given. Then there is a
constant $C_{\rho, r}>0$ such that for any $N>8^{2r/\rho}$ and for any
set $\Lambda$ as in \eqref{eq:300}
and for every $f\in\ell^{2r}\big(\ZZ^d\big)$ we have  
\begin{align}
  \label{eq:677}
  \Big\|\sum_{u\in \mathcal U_{\NN_k}}f_u\Big\|_{\ell^{2r}}\le C_{\rho, r}\|f\|_{\ell^{2r}}.
\end{align}
Moreover, the integer $k\in\NN_D$, the set $\mathcal U_{\NN_k}$ and consequently  the sets $S_1,\ldots, S_k$
 are determined by the set $\Lambda$ as it was described  above.
\end{theorem}

 The
estimate \eqref{eq:677} will follow from Theorem \ref{thm:5} and
Proposition \ref{prop:5} formulated below. 
Let us introduce a suitable square function which will
be useful in bounding  \eqref{eq:677}. 
For any $M\subseteq \NN_k$ and $L=\{j_1, \ldots, j_l\}\subseteq M$ and
any sequence
  $\sigma=(s_{j_1},\ldots,s_{j_l})\in
  \NN_{\beta_{j_1}}\times\ldots\times\NN_{\beta_{j_l}}$ determined by the
  set $L$ let
us define the following square function 
$\mathcal S_{L, M}^{\sigma}\big(f_u:u\in\mathcal
  U_{\NN_k}\big)$  
associated with the sequence $\big(f_u:u\in\mathcal
  U_{\NN_k}\big)$ of complex-valued functions as in \eqref{eq:25}, by setting

\begin{align}
  \label{eq:44}
  \mathcal S_{L, M}^{\sigma}\big(f_u(x):u\in\mathcal
  U_{\NN_k}\big)=\bigg(
\sum_{w\in\mathcal U_{M^c}}\Big|\sum_{u\in\mathcal U_{M\setminus
    L}}\sum_{v\in\mathcal V_{L}^{\sigma}}f_{w+u+v}(x)\Big|^2\bigg)^{1/2},
\end{align}
where $M^c=\NN_k\setminus M$.
For some $s_{j_i}\in\{s_{j_1}, \ldots, s_{j_l}\}$ we will write
\begin{align*}
  \big\|\mathcal S_{L, M}^{\sigma}\big(f_u(x):u\in\mathcal
  U_{\NN_k}\big)\big\|_{\ell^2_{s_{j_i}}}=\bigg(\sum_{s_{j_i}\in\NN_{\beta_{j_i}}}
\big|\mathcal S_{L, M}^{(s_{j_1},\ldots,s_{j_l})}\big(f_u:u\in\mathcal
  U_{\NN_k}\big)(x)\big|^2\bigg)^{1/2}
\end{align*}
which defines some function which depends on $x\in\ZZ^d$ and on each
$s_{j_n}\in \{s_{j_1}, \ldots, s_{j_l}\}\setminus\{s_{j_i}\}$.

 For the proof of
\eqref{eq:677} we have to exploit the fact that the Fourier transform
of $f_u$ is defined as a sum of disjointly supported smooth cut-off
functions. Then appropriate subsums of
$\sum_{u\in\mathcal U_{\NN_k}}f_u$ should be strongly orthogonal to
each other. 

Theorem \ref{thm:500} will be proved as a consequence of Theorem
\ref{thm:5} and Proposition \ref{prop:5} below.  

\begin{theorem}
  \label{thm:5}
Suppose that $\rho>0$ and $r\in\NN$ are given. Then there is a
constant $C_{\rho, r}>0$ such that for any $N>8^{2r/\rho}$ and for any
set $\Lambda$ as in \eqref{eq:300}
and for every $f\in\ell^{2r}\big(\ZZ^d\big)$ we have  
\begin{align}
  \label{eq:130}
  \Big\|\sum_{u\in\mathcal U_{\NN_k}}
  f_{u}\Big\|_{\ell^{2r}}^{2r}
\le C_{\rho, r}\sum_{\atop{M\subseteq \NN_k}{M=\{j_1, \ldots,
    j_m\}}}\sum_{\sigma\in\NN_{\beta_{j_1}}\times\ldots\times\NN_{\beta_{j_m}}}
\big\|\mathcal S_{M, M}^{\sigma}\big(f_u:u\in\mathcal
  U_{\NN_k}\big)\big\|_{\ell^{2r}}^{2r}.
\end{align}
Moreover, the integer $k\in\NN_D$, the set $\mathcal U_{\NN_k}$ and consequently  the sets $S_1,\ldots, S_k$
 are determined by the set $\Lambda$ as it was described  above the
 formulation of
 Theorem \ref{thm:500}.
\end{theorem}
\begin{proof}

Under the assumptions of Theorem \ref{thm:500}, there is a constant
$C_r>0$ such that for any $M\subseteq\NN_k$ and $L=\{j_1,\ldots,
j_l\}\subseteq M$ and $j_n\in M\setminus L$ and for any
$\sigma=(s_{j_1}, \ldots,
s_{j_l})\in\NN_{\beta_{j_1}}\times\ldots\times\NN_{\beta_{j_l}}$
determined by the set $L$ we have
\begin{align}
  \label{eq:119}
 \big\| \mathcal S_{L, M}^{\sigma}\big(f_u:u\in\mathcal
  U_{\NN_k}\big)\big\|_{\ell^{2r}}\le C_r\Big\|\big\|\mathcal
  S_{L\cup\{j_n\}, M}^{\sigma\oplus s_{j_n}}\big(f_u:u\in\mathcal
  U_{\NN_k}\big)\big\|_{\ell^2_{s_{j_n}}}\Big\|_{\ell^{2r}}
\end{align}
where $\sigma\oplus s_{j_n}=(s_{j_1},\ldots,s_{j_l},
s_{j_n})\in
\NN_{\beta_{j_1}}\times\ldots\times\NN_{\beta_{j_l}}\times\NN_{\beta_{j_n}}$
is the sequence determined by the set
$L\cup\{s_{j_n}\}$. Moreover, the right-hand side in \eqref{eq:119}
can be controlled in the following way
\begin{multline}
  \label{eq:129}
  \Big\|\big\|\mathcal
  S_{L\cup\{j_n\}, M}^{\sigma\oplus s_{j_n}}\big(f_u:u\in\mathcal
  U_{\NN_k}\big)\big\|_{\ell^2_{ s_{j_n}}}\Big\|_{\ell^{2r}}^{2r}
\le C_r\sum_{s_{j_n}\in\NN_{\beta_{j_n}}}
 \big\|\mathcal S_{L\cup\{j_n\}, M}^{\sigma\oplus s_{j_n}}\big(f_u:u\in\mathcal
  U_{\NN_k}\big)\big\|_{\ell^{2r}}^{2r}\\
+C_r\big\|\mathcal S_{L, M\setminus\{j_n\}}^{\sigma}\big(f_u:u\in\mathcal
  U_{\NN_k}\big)\big\|_{\ell^{2r}}^{2r}.
\end{multline}
The proof of \eqref{eq:119} and \eqref{eq:129} can be found in
\cite{mst1}. Therefore, \eqref{eq:119} combined with 
  \eqref{eq:129} yields
\begin{align}
  \label{eq:1190}
\begin{split}
 \big\| \mathcal S_{L, M}^{\sigma}\big(f_u:u\in\mathcal
  U_{\NN_k}\big)\big\|_{\ell^{2r}}^{2r}&\le 
C_r\sum_{s_{j_n}\in\NN_{\beta_{j_n}}}
 \big\|\mathcal S_{L\cup\{j_n\}, M}^{\sigma\oplus s_{j_n}}\big(f_u:u\in\mathcal
  U_{\NN_k}\big)\big\|_{\ell^{2r}}^{2r}\\
&+C_r\big\|\mathcal S_{L, M\setminus\{j_n\}}^{\sigma}\big(f_u:u\in\mathcal
  U_{\NN_k}\big)\big\|_{\ell^{2r}}^{2r}.
\end{split}
\end{align}
Applying \eqref{eq:1190} recursively we obtain 

\begin{multline}
  \label{eq:131}
\Big\|\sum_{u\in\mathcal U_{\NN_k}}
  f_{u}\Big\|_{\ell^{2r}}^{2r}=\big\|\mathcal S_{\emptyset,
    \NN_k}\big(f_u: u\in\mathcal
  U_{\NN_k}\big)\big\|_{\ell^{2r}}^{2r}\\
\lesssim _r  
\sum_{s_{k}\in\NN_{\beta_{j_k}}}
 \big\|\mathcal S_{\{k\}, \NN_k}^{(s_{k})}\big(f_u:u\in\mathcal
  U_{\NN_k}\big)
\big\|_{\ell^{2r}}^{2r}
+\big\|\mathcal S_{\emptyset, \NN_{k-1}}\big(f_u:u\in\mathcal
  U_{\NN_k}\big)
\big\|_{\ell^{2r}}^{2r}\\
\lesssim_r
\sum_{s_{k-1}\in\NN_{\beta_{j_{k-1}}}}\sum_{s_{k}\in\NN_{\beta_{j_k}}}
 \big\|\mathcal S_{\{k-1, k\}, \NN_k}^{(s_{k-1}, s_{k})}\big(f_u:u\in\mathcal
  U_{\NN_k}\big)
\big\|_{\ell^{2r}}^{2r}+
\sum_{s_{k}\in\NN_{\beta_{j_k}}}
 \big\|\mathcal S_{\{k\}, \NN_{k}\setminus\{k-1\}}^{(s_{k})}\big(f_u:u\in\mathcal
  U_{\NN_k}\big)
\big\|_{\ell^{2r}}^{2r}\\
+\sum_{s_{k-1}\in\NN_{\beta_{j_{k-1}}}}
 \big\|\mathcal S_{\{k-1\}, \NN_{k-1}}^{(s_{k-1})}\big(f_u:u\in\mathcal
  U_{\NN_k}\big)
\big\|_{\ell^{2r}}^{2r}
+\big\|\mathcal S_{\emptyset, \NN_{k-2}}\big(f_u:u\in\mathcal
  U_{\NN_k}\big)
\big\|_{\ell^{2r}}^{2r}\\
\lesssim_r\ldots\lesssim_{\rho, r}\sum_{\atop{M\subseteq \NN_k}{M=\{j_1, \ldots,
    j_m\}}}\sum_{\sigma\in\NN_{\beta_{j_1}}\times\ldots\times\NN_{\beta_{j_m}}}
\sum_{x\in\ZZ^d}\bigg(\sum_{w\in\mathcal
  U_{M^c}}\Big|\sum_{v\in\mathcal V_{M}^{\sigma}}f_{w+v}(x)\Big|^2\bigg)^r.
\end{multline}
 The proof of \eqref{eq:130} is completed.
\end{proof}
\subsection{Concluding remarks and the proof of Theorem \ref{thm:500}}
Now Theorem \ref{thm:5} reduces the proof of inequality \eqref{eq:677}
to showing the following estimate 
\begin{align}
  \label{eq:134}
\sum_{\atop{M\subseteq \NN_k}{M=\{j_1, \ldots,
    j_m\}}}\sum_{\sigma\in\NN_{\beta_{j_1}}\times\ldots\times\NN_{\beta_{j_m}}}
\big\|\mathcal S_{M, M}^{\sigma}\big(f_u:u\in\mathcal
  U_{\NN_k}\big)\big\|_{\ell^{2r}}^{2r}\lesssim_r\|f\|_{\ell^{2r}}^{2r}
\end{align}
for any
$f\in\ell^{2r}\big(\ZZ^d\big)$ which is a
characteristic function of a finite set in $\ZZ^d$.
Firstly, we prove the following.
\begin{proposition}
  \label{prop:5}
Under the assumptions of Theorem \ref{thm:500}, 
there exists a constant $C_{\rho, r}>0$
  such that for any $M=\{j_1,\ldots, j_m\}\subseteq\NN_{k}$ any $\sigma=(s_{j_1},\ldots,
s_{j_m})\in\NN_{\beta_{j_1}}\times\ldots\times\NN_{\beta_{j_m}}$
determined by the set $M$ and $f\in\ell^{2r}\big(\ZZ^d\big)$ we
have
\begin{equation}
  \label{eq:138}
  \big\|\mathcal S_{M,
    M}^{\sigma}\big(f_u:u\in\mathcal
  U_{\NN_k}\big)\big\|_{\ell^{2r}}
\le C_{\rho, r} \mathbf A_r \Big\|\mathcal S_{M,
    M}^{\sigma}\Big(\mathcal
  F^{-1}\Big(\sum_{b\in\NN_{Q_0}}\eta_N(\xi-b/Q_0-u)\hat{f}(\xi)\Big):u\in\mathcal
  U_{\NN_k}\Big)\Big\|_{\ell^{2r}}.
\end{equation}
\end{proposition}
\begin{proof}
We assume, without of loss of generality, that $N\in\NN$ is large.
Let $B_h=q_{j_1, s_{j_1}}\cdot\ldots\cdot q_{j_m,
  s_{j_m}}\cdot Q_0\le e^{N^{\rho}}$
and observe that according to the notation from \eqref{eq:25} and \eqref{eq:13},  we
have
\begin{multline}
  \label{eq:136}
  \big\|\mathcal S_{M, M}^{\sigma}\big(f_u:u\in\mathcal
  U_{\NN_k}\big)\big\|_{\ell^{2r}}^{2r}=
\sum_{x\in\ZZ^d}\bigg(\sum_{w\in\mathcal
  U_{M^c}}\Big|\sum_{v\in\mathcal V_{M}^{\sigma}}f_{w+v}(x)\Big|^2\bigg)^r\\
\le\sum_{x\in\ZZ^d}\bigg(\sum_{w\in\mathcal
  U_{M^c}}\Big|\calF^{-1}\Big(\sum_{v\in\mathcal V_{M}^{\sigma}}\sum_{b\in\NN_{Q_0}}\Theta(\xi-b/Q_0-v-w)\eta_N(\xi-b/Q_0-v-w)\hat{f}(\xi)\Big)(x)\Big|^2\bigg)^r\\
=\sum_{n\in \NN_{B_h}^d}\sum_{x\in\ZZ^d}\bigg(\sum_{w\in\mathcal
  U_{M^c}}\big|\calF^{-1}\big(\Theta\eta_N G(\: \cdot\: ;n, w)\big)(B_hx+n)\big|^2\bigg)^r
\end{multline}
where 
\begin{align}
  \label{eq:137}
  G(\xi; n, w)=\sum_{v\in\mathcal V_{M}^{\sigma}}\sum_{b\in\NN_{Q_0}^d}\hat{f}(\xi+b/Q_0+v+w)e^{-2\pi i
(b/Q_0+v)\cdot n}.
\end{align}

We  know that for each $0<p<\infty$ there is a constant $C_p>0$ such
that for any $d\in\NN$ and $\lambda_1,\ldots,\lambda_d\in\CC^d$ we have
\begin{align}
  \label{eq:1411}
  \bigg(\int_{\CC^d}|\lambda_1z_1+\ldots+\lambda_dz_d|^pe^{-\pi
  |z|^2} {\rm d} z\bigg)^{1/p}=C_p\big(|\lambda_1|^2+\ldots+|\lambda_d|^2\big)^{1/2}.
\end{align}

By Proposition 4.5 from \cite{mst1}, with the sequence of multipliers
$\Theta_N=\Theta$ for all $N\in\NN$ and $\Theta$ as in \eqref{eq:155},   we have 
\begin{align}
\label{eq:140}
\big\|\calF^{-1}\big(\Theta\eta_N G(\: \cdot\: ;n,
w)\big)(B_hx+n)\big\|_{\ell^{2r}(x)}\le C_{\rho, r}\mathbf A_{2r}
\big\|\calF^{-1}\big(\eta_N G(\: \cdot\: ;n, w)\big)(B_hx+n)\big\|_{\ell^{2r}(x)}
\end{align}
since $\inf_{\gamma\in\Gamma}\varepsilon_{\gamma}^{-1}\ge
e^{N^{2\rho}}\ge 2e^{(d+1)N^{\rho}}\ge B_h$ for sufficiently large
$N\in\NN$. 

Therefore, combining \eqref{eq:140} with \eqref{eq:1411} we obtain that
\begin{multline}
  \label{eq:142}
\sum_{n\in \NN_{B_h}^d}\sum_{x\in\ZZ^d}\bigg(\sum_{w\in\mathcal
  U_{M^c}}\big|\calF^{-1}\big(\Theta\eta_N G(\: \cdot\: ;n, w)\big)(B_hx+n)\big|^2\bigg)^r
\\=C_{2r}^{2r}\int_{\CC^d}
\sum_{n\in \NN_{B_h}^d}\sum_{x\in\ZZ^d}\Big|\calF^{-1}\Big(\Theta\eta_N\Big(\sum_{w\in\mathcal
	U_{M^c}}z_{w} G(\: \cdot\: ;n, w)\Big)\Big)(B_hx+n)\Big|^{2r}e^{-\pi |z|^2}{\rm d}z\\
\lesssim_r\int_{\CC^d}
\sum_{n\in \NN_{B_h}^d}\sum_{x\in\ZZ^d}\Big|\calF^{-1}\Big(\sum_{w\in\mathcal
	U_{M^c}}z_{w} \eta_NG(\: \cdot\: ;n, w)\Big)(B_hx+n)\Big|^{2r}e^{-\pi |z|^2}{\rm d}z\\
\lesssim_r\sum_{n\in \NN_{B_h}^d}\sum_{x\in\ZZ^d}\bigg(\sum_{w\in\mathcal
  U_{M^c}}\big|\calF^{-1}\big(\eta_N G(\: \cdot\: ;n, w)\big)(B_hx+n)\big|^2\bigg)^r\\
\lesssim_r\sum_{x\in\ZZ^d}\bigg(\sum_{w\in\mathcal
  U_{M^c}}\Big|\calF^{-1}\Big(\sum_{v\in\mathcal V_{M}^{\sigma}}\sum_{b\in\NN_{Q_0}}\eta_N(\xi-b/Q_0-v-w)\hat{f}(\xi)\Big)(x)\Big|^2\bigg)^r.
\end{multline}
This completes the proof of Proposition \ref{prop:5}.
\end{proof}

Now we are able to finish the proof of Theorem
\ref{thm:500}. 
\begin{proof}[Proof of Theorem \ref{thm:500}]
  It remains to show that  there exists a constant $C_{\rho, r}>0$
  such that for any $M=\{j_1,\ldots, j_m\}\subseteq\NN_{k}$ any $\sigma=(s_{j_1},\ldots,
s_{j_m})\in\NN_{\beta_{j_1}}\times\ldots\times\NN_{\beta_{j_m}}$
determined by the set $M$ and $f\in\ell^{2r}\big(\ZZ^d\big)$ we
have
\begin{align}
  \label{eq:143}
\sum_{\sigma\in\NN_{\beta_{j_1}}\times\ldots\times\NN_{\beta_{j_m}}}\Big\|\mathcal S_{M,
    M}^{\sigma}\Big(\mathcal
  F^{-1}\Big(\sum_{b\in\NN_{Q_0}}\eta_N(\xi-b/Q_0-u)\hat{f}(\xi)\Big):u\in\mathcal
  U_{\NN_k}\Big)\Big\|_{\ell^{2r}}^{2r}\le C_{\rho, r}^{2r}\|f\|_{\ell^{2r}}.
\end{align}
Since there are $2^k$ possible choices of sets $M\subseteq \NN_k$ and $k\in\NN_D$
then \eqref{eq:134} will follow and the proof of Theorem \ref{thm:500}
will be completed. If $r=1$ the then Plancherel's theorem does the job
since the  functions $\eta_N(\xi-b/Q_0-v-w)$ are disjointly supported 
for all $b/Q_0\in \NN_{Q_0}$, $w\in\mathcal
  U_{M^c}$, $v\in\mathcal V_{M}^{\sigma}$ and $\sigma=(s_{j_1},\ldots,
s_{j_m})\in\NN_{\beta_{j_1}}\times\ldots\times\NN_{\beta_{j_m}}$. For
general $r\ge2$, since $\|f\|_{\ell^{2r}}^{2r}=\|f\|_{\ell^{2}}^2$
because we have assumed that $f$ is a characteristic function of a
finite set in $\ZZ^d$, it
suffices to prove for any $x\in\ZZ^d$ that
\begin{align}
  \label{eq:139}
  \sum_{w\in\mathcal
  U_{M^c}}\Big|\calF^{-1}\Big(\sum_{v\in\mathcal V_{M}^{\sigma}}\sum_{b\in\NN_{Q_0}}\eta_N(\xi-b/Q_0-v-w)\hat{f}(\xi)\Big)(x)\Big|^2\le
C_{\rho, r}.
\end{align}
In fact, since $\|f\|_{\ell^{\infty}}=1$, it is enough to show 
\begin{align}
  \label{eq:144}
  \Big\|\calF^{-1}\Big(\sum_{w\in\mathcal
  U_{M^c}}\alpha(w)\sum_{v\in\mathcal V_{M}^{\sigma}}\sum_{b\in\NN_{Q_0}}\eta_N(\xi-b/Q_0-v-w)\Big)\Big\|_{\ell^1}\le
C_{\rho, r}
\end{align}
for any sequence of complex numbers $\big(\alpha(w):  w\in\mathcal
  U_{M^c}\big)$ such that
  \begin{align}
    \label{eq:50}
    \sum_{w\in\mathcal
  U_{M^c}}|\alpha(w)|^2=1.
  \end{align}
Computing the Fourier transform we obtain 
\begin{multline}
  \label{eq:145}
  \calF^{-1}\Big(\sum_{w\in\mathcal
    U_{M^c}}\alpha(w)\sum_{v\in\mathcal V_{M}^{\sigma}}\sum_{b\in\NN_{Q_0}}\eta_N(\xi-b/Q_0-v-w)\Big)\\
  =\Big(\sum_{w\in\mathcal U_{M^c}}\alpha(w)e^{-2\pi i x\cdot
    w}\Big)\cdot \det(\mathcal E_N)\mathcal
  F^{-1}\eta\big(\mathcal E_Nx\big)\cdot\Big(\sum_{v\in\mathcal
    V_{M}^{\sigma}}\sum_{b\in\NN_{Q_0}}e^{-2\pi i
    x\cdot (b/Q_0+v)}\Big).
\end{multline}
The function 
\begin{align}
  \label{eq:146}
  \sum_{v\in\mathcal
    V_{M}^{\sigma}}\sum_{b\in\NN_{Q_0}}e^{-2\pi i
    x\cdot (b/Q_0+v)}
\end{align}
 can be written as a sum of $2^m$ functions 
\begin{align}
  \label{eq:147}
  \sum_{b\in\NN_{Q}}e^{-2\pi i
    x\cdot (b/Q)} =
\begin{cases} 
Q^d & \mbox{if } x \equiv 0 \pmod{Q}, \\
0   & \mbox{otherwise,} 
\end{cases}  
\end{align}
where possible values of $Q$ are products of $Q_0$ and $p_{j_i,
  s_{j_i}}^{\gamma_i}$ or $p_{j_i, s_{j_i}}^{\gamma_i-1}$ for
$i\in\NN_m$. Therefore, the proof of \eqref{eq:144} will be completed
if we show  that
\begin{align}
  \label{eq:148}
  \Big\|\Big(\sum_{w\in\mathcal U_{M^c}}\alpha(w)e^{-2\pi i Qx\cdot
    w}\Big)\cdot Q^d\det(\mathcal E_N)\mathcal
  F^{-1}\eta\big(Q\mathcal E_N x\big)\Big\|_{\ell^1(x)}\le C_{\rho, r}
\end{align}
 for any integer $Q\le e^{N^{\rho}}$ such that
$(Q, q_{j, s})=1$, for all $j\in M^c$ and $s\in\NN_{\beta_j}$. 

Recall that, according Remark \ref{rem:1}, in our case $\eta=\phi*\psi$ for some two smooth functions
$\phi, \psi$
supported in $(-1/(8d), 1/(8d))^d$.  Therefore, by the Cauchy--Schwarz inequality we only need to prove
that
\begin{align}
  \label{eq:22}
Q^{d/2}\det(\mathcal E_N)^{1/2}\big\|\mathcal
  F^{-1}\phi\big(Q\mathcal E_N x\big)\big\|_{\ell^2(x)}\le C_{\rho, r}  
\end{align}
and
\begin{align}
  \label{eq:15}
Q^{d/2}\det(\mathcal E_N)^{1/2}\Big\|\Big(\sum_{w\in\mathcal U_{M^c}}\alpha(w)e^{-2\pi i Qx\cdot
    w}\Big)\cdot\mathcal
  F^{-1}\psi\big(Q\mathcal E_N x\big)\Big\|_{\ell^2(x)}\le C_{\rho, r}.
\end{align}
Since $(Q, q_{j, s})=1$, for all $j\in M^c$ and $s\in\NN_{\beta_j}$
then $Qw\not\in \ZZ^d$ for any $w\in\mathcal U_{M^c}$  and its
denominator is bounded by $N^D$. We can assume, without of loss of
generality, that $Qw\in[0, 1)^d$ by the periodicity of  $x\mapsto e^{-2\pi i x\cdot
    Qw}$.
Inequality \eqref{eq:22} easily follows from Plancherel's theorem. In
order to prove \eqref{eq:15} observe that by the change of variables
one has
\begin{align*}
  \Big(\sum_{w\in\mathcal U_{M^c}}\alpha(w)e^{-2\pi i x\cdot
    Qw}\Big)\cdot\mathcal
  F^{-1}\psi\big(Q\mathcal E_N x\big)=Q^{-d}\det(\mathcal E_N)^{-1}
  \sum_{w\in\mathcal U_{M^c}}\alpha(w)\mathcal
  F^{-1}\big(\psi\big(Q^{-1}\mathcal E_N^{-1}(\: \cdot-Qw) \big)\big)(x).
\end{align*}
 Therefore, Plancherel's theorem and the last identity yield  
 \begin{multline}
   \label{eq:58}
   Q^{d}\det(\mathcal E_N)\Big\|\Big(\sum_{w\in\mathcal U_{M^c}}\alpha(w)e^{-2\pi i Qx\cdot
    w}\Big)\cdot\mathcal
  F^{-1}\psi\big(Q\mathcal E_N x\big)\Big\|_{\ell^2(x)}^2\\=
\sum_{w\in\mathcal
  U_{M^c}}|\alpha(w)|^2\int_{\RR^d}\big|\psi(\xi-\mathcal E_N^{-1}w)\big|^2{\rm
d}\xi+\sum_{\atop{w_1, w_2\in\mathcal
  U_{M^c}}{w_1\not=w_2}}\alpha(w_1)\overline{\alpha(w_2)}\int_{\RR^d}\psi(\xi)\psi\big(\xi-\mathcal E_N^{-1}(w_1-w_2)\big){\rm
d}\xi.
 \end{multline}
The first sum on the right-hand side of \eqref{eq:58} is bounded in
view of \eqref{eq:50}. The second one vanishes since the function
$\psi$ is supported in $(-1/(8d), 1/(8d))^d$ and
$|\mathcal E_N^{-1}(w_1-w_2)|_{\infty}\ge e^{N^{2\rho}}N^{-2D}>1$,
for sufficiently large $N$.
The proof of Theorem \ref{thm:500} is completed.
\end{proof}

\section{Proof of Theorem \ref{thm:7}}
\label{sec:4}
This section is intended to provide the proof of Theorem \ref{thm:7}. 
In fact, in view of the decomposition of the kernel $K$ into dyadic
pieces as in \eqref{eq:87}, instead
of inequality \eqref{eq:37}, it suffices to show that for every $p\in(1, \infty)$
there is a constant $C_p>0$ such that for all
$f\in\ell^p\big(\ZZ^d\big)$ we have
\begin{align}
  \label{eq:84}
\Big\|\sum_{n\in\ZZ}T_nf\Big\|_{\ell^p}\le C_p\|f\|_{\ell^p}  
\end{align}
where 
\begin{align}
  \label{eq:85}
T_n f(x)
=
\sum_{y\in\ZZ^k} f\big(x - \calQ(y)\big) K_n(y)    
\end{align}
with the kernel $K_n$ as in \eqref{eq:87} for each $n\in\ZZ$.

\subsection{Exponential sums and $\ell^2\big(\ZZ^d\big)$ approximations}

 Recall that for $q\in\NN$ 
\[
A_q=\big\{a\in\NN_q^d: \mathrm{gcd}\big(q, \mathrm(a_{\gamma}: \gamma\in\Gamma)\big)\big\}.
\]
Now for $q \in \NN$ and $a \in A_q$ we define the Gaussian sums 
$$
G(a/q) = q^{-k} \sum_{y \in \NN^k_q} e^{2\pi i \sprod{(a/q)}{\calQ(y)}}.
$$
Let us observe that there exists $\delta>0$ such that
\begin{equation}
	\label{eq:20}
	\lvert G(a/q) \rvert \lesssim q^{-\delta}.
\end{equation}
This follows from  the multi-dimensional variant of
Weyl's inequality (see \cite[Proposition 3]{SW0}).

 Let $P$ be a polynomial in $\RR^k$ of degree $d \in\NN$ such that
\[
	P(x) = \sum_{0 < \norm{\gamma} \leq d} \xi_\gamma x^\gamma.
\]
Given $N \geq 1$, let $\Omega_N$ be a convex set in $\RR^k$ such that
\[
	 \Omega_N \subseteq \big\{x \in \RR^k : \norm{x-x_0} \leq cN \big\}
\]
for some $x_0\in\RR^k$ and $c > 0$.
 We define the Weyl sums
\begin{align}
  \label{eq:215}
	S_N = \sum_{n \in \Omega_N \cap \ZZ^k} e^{2\pi i P(n)}\varphi(n)
\end{align}
where  $\varphi:\RR^k\mapsto \CC$ is  a continuously differentiable function which for some $C>0$ satisfies
\begin{align}
  \label{eq:217}
|\varphi(x)|\le C, \qquad \text{and} \qquad |\nabla \varphi(x)|\le C(1+|x|)^{-1}.  
\end{align}

In \cite{mst1} we have proven Theorem \ref{thm:3} which is a  refinement of the estimates for the
multi-dimensional Weyl sums $S_N$, where the limitations $N^\varepsilon \leq q \leq
  N^{k-\varepsilon}$ from
  \cite[Proposition 3]{SW0} are replaced by the weaker restrictions
  $(\log{N})^\beta \leq q \leq N^k(\log{N})^{-\beta}$ for 
  appropriate $\beta$. Namely.

\begin{theorem}
	\label{thm:3}
        Assume that there is a multi-index $\gamma_0$ such that $0 < \norm{\gamma_0} \leq d$ and
	\[
		\Big\lvert
		\xi_{\gamma_0} - \frac{a}{q}
		\Big\rvert
		\leq
		\frac{1}{q^2}
	\]
        for some integers $a, q$ such that $0\le a\le q$ and $(a, q) =
        1$. Then for any $\alpha>0$ there is  $\beta_{\alpha}>0$ so that, for any
        $\beta\ge \beta_{\alpha}$, if
	\begin{equation}
		\label{eq:57}
		(\log N)^\beta \leq q \leq N^{\norm{\gamma_0}} (\log N)^{-\beta}
	\end{equation}
        then there is a constant $C>0$ 
	\begin{equation}
		\label{eq:56}
		|S_N|
		\leq
		C
		N^k (\log N)^{-\alpha}.
	\end{equation}
	The implied  constant $C$ is independent of $N$. 
\end{theorem}

Let $\big(\seq{m_n}{n \geq 0}\big)$ be a sequence of
multipliers on $\TT^d$, corresponding to the operators
\eqref{eq:85}.  
Then for any finitely supported function  $f:\ZZ^d\mapsto\CC$ we see that
$$
T_nf(x)=\calF^{-1} \big(m_n \hat{f}\big)(x)
$$
where 
$$
m_n(\xi) = \sum_{y \in \ZZ^k} e^{2\pi i \sprod{\xi}{\calQ(y)}} K_n(y).
$$
For $n \geq 0$ we set
$$
\Phi_n(\xi) = \int_{\RR^k} e^{2\pi i \sprod{\xi}{\calQ(y)}} K_n(y) {\rm d}y.
$$
Using multi-dimensional version of van der Corput's lemma (see \cite[Propositon 2.1]{sw})
we obtain
\begin{equation}
	\label{eq:18}
	\lvert \Phi_n(\xi) \rvert
	\lesssim
	\min\big\{1, \abs{2^{n A} \xi}_{\infty}^{-1/d}\big\}.
\end{equation}
Moreover, if $n \geq 1$ we have
\begin{equation}
	\label{eq:19}
	\lvert \Phi_n(\xi) \rvert
	=
	\Big\lvert
	\Phi_n(\xi) - \int_{\RR^k} K_n(y) {\rm d}y
	\Big\rvert
	\lesssim
	\min\big\{1, \abs{2^{nA} \xi}_{\infty} \big\}.
\end{equation}

The next proposition shows relations between $m_n$ and $\Phi_n$. 

\begin{proposition}
  \label{prop:0}
 There is a constant $C>0$ such
  that for every $n\in\NN$ and for every $\xi\in [1/2, 1/2)^d$ satisfying 
        $$
	\Big\lvert \xi_\gamma - \frac{a_\gamma}{q} \Big\rvert \leq
        L_1^{-|\gamma|}L_2
	$$
	for all $\gamma \in \Gamma$, where  $1\le q\le L_3\le 2^{n/2}$, $a\in
        A_q$, $L_1\ge 2^n$  and $L_2\ge1$ we have 
        \begin{align}
          \label{eq:41}
          \big|m_n(\xi)-G(a/q)\Phi_{n}(\xi-a/q)\big|\le 
          C\Big(L_32^{-n}
          +L_2L_32^{-n}\sum_{\gamma \in
            \Gamma}\big(2^n/L_1\big)^{|\gamma|}\Big)\le CL_2L_32^{-n}.
        \end{align}

\end{proposition}
\begin{proof}
	Let $\theta = \xi - a/q$. For any $r \in \NN_q^k$, if $y \equiv r \pmod q$ then
	for each $\gamma \in \Gamma$
	$$
	\xi_\gamma y^\gamma \equiv \theta_\gamma y^\gamma
	+ (a_\gamma/q) r^\gamma \pmod 1,
	$$
	thus
	$$
	\sprod{\xi}{\calQ(y)} \equiv \sprod{\theta}{\calQ(y)} + \sprod{(a/q)}{\calQ(r)} \pmod 1.
	$$
	Therefore,
	\[
	\sum_{y \in \ZZ^k} e^{2\pi i \sprod{\xi}{\calQ(y)}} K_n(y)
	=
	\sum_{r \in \NN_q^k}
	e^{2\pi i \sprod{(a/q)}{\calQ(r)}}
	\sum_{y \in \ZZ^k}
	e^{2\pi i \sprod{\theta}{\calQ(qy+r)}}
	K_n(qy+r).
	\]
	If $2^{n-2} \leq \norm{q y + r}, \norm{qy} \leq 2^n$ then by the
        mean value theorem we obtain
	\[
	\big\lvert
	\sprod{\theta}{\calQ(q y + r)} - \sprod{\theta}{\calQ(q y)}
	\big\rvert
	\lesssim
	\norm{r}
	\sum_{\gamma \in \Gamma}
	\abs{\theta_\gamma}
	\cdot
	2^{n(\abs{\gamma} - 1)}
	\lesssim
	q \sum_{\gamma \in \Gamma}
	L_1^{-\abs{\gamma}} L_2 2^{n(\abs{\gamma}-1)}
	\lesssim L_2L_32^{-n}\sum_{\gamma \in \Gamma}\big(2^n/L_1\big)^{\abs{\gamma}}
	\]
	and
	$$
	\big\lvert
	K_n(q y + r) - K_n(q y)
	\big\rvert
	\lesssim
	2^{-n(k+1)}L_3.
        $$
	Thus
        \[
        \sum_{y \in \ZZ^k} e^{2\pi i \sprod{\xi}{\calQ(y)}} K_n(y)
	=G(a/q)\cdot q^k\sum_{y \in \ZZ^k} e^{2\pi i
          \sprod{\theta}{\calQ(qy)}} K_n(qy)
+\mathcal O\Big(L_32^{-n}+L_2L_32^{-n}\sum_{\gamma \in \Gamma}\big(2^n/L_1\big)^{|\gamma|}\Big).
        \]
Now one can replace the sum on the right-hand side by the
integral. Indeed,  again by the mean value theorem  we obtain
  \begin{multline*}
    \Big|\sum_{y \in \ZZ^k} e^{2\pi i
          \sprod{\theta}{\calQ(qy)}} K_n(qy)-\int_{\RR^k}e^{2\pi i
          \theta\cdot\calQ(qt)}K_n(qt){\rm d}t\Big|\\
    = \Big|\sum_{y \in \ZZ^k} \int_{[0, 1)^k}\big(e^{2\pi i
          \sprod{\theta}{\calQ(qy)}} K_n(qy)-e^{2\pi i
          \theta\cdot\calQ(q(y+t))}K_n(q(y+t)){\rm d}t\big)\Big|\\
=\mathcal O\Big(q^{-k}L_32^{-n}+q^{-k}L_2L_32^{-n}\sum_{\gamma \in \Gamma}\big(2^n/L_1\big)^{|\gamma|}\Big).
  \end{multline*}
This completes the proof of Proposition \ref{prop:0}.
\end{proof}

\subsection{Discrete Littlewood--Paley theory}
Fix $j, n\in\ZZ$ and $N\in\NN$ and let $\mathcal
E_N$ be a diagonal $d\times d$ matrix with positive entries
$(\varepsilon_{\gamma}: \gamma\in\Gamma)$ such that
$\varepsilon_{\gamma} \le e^{-N^{2\rho}}$ with $\rho>0$ as in Section \ref{sec:6}. Let us consider the multipliers 
\begin{align}
  \label{eq:9}
  \Omega_{N}^{j, n}(\xi)=\sum_{a/q\in\mathscr U_{N}}\Phi_{j,
    n}(\xi-a/q)\eta_N(\xi-a/q)
\end{align}
with $\eta_N(\xi)=\eta\big(\mathcal E_N^{-1}\xi\big)$ and  $\Phi_{j,
  n}(\xi)=\Phi\big(2^{nA+jI}\xi\big)$, where $\Phi$
is  a Schwartz function such that $\Phi(0)=0$.
  
If we had $\mathscr U_{N}=\{0\}$ then $\Omega_{N}^{j, n}(\xi)$ could
be treated as a standard Littlewood--Paley projector.  Now we
formulate an abstract  theorem which can be thought as a discrete
variant of Littlewood--Paley theory. Its proof will be based on
Theorem \ref{th:3}. Here we only obtain some square function estimate
which is interesting in its own right. However, we will be able appreciate its
usefulness later, in the proof of inequality \eqref{eq:84}.
\begin{theorem}
  \label{thm:30}
For every
$p\in(1, \infty)$ there is a
constant $C_p>0$ such that for all $-\infty\le M_1\le M_2\le \infty$,
$j\in\ZZ$ and $N\in\NN$ and every  $f\in\ell^p\big(\ZZ^d\big)$ we have
 \begin{align}
   \label{eq:23}
   \bigg\|\Big(\sum_{M_1\le n\le M_2}\big|\mathcal
  F^{-1}\big(\Omega_{N}^{j, n}\hat{f}\big)\big|^2\Big)^{1/2}\bigg\|_{\ell^p}\le C_p \log N\|f\|_{\ell^p}.
 \end{align}
\end{theorem}
\begin{proof}
  
By Khinchine's inequality \eqref{eq:23} is equivalent to the following
\begin{align}
  \label{eq:24}
  \bigg(\int_0^1\bigg\|\sum_{M_1\le n\le M_2}\varepsilon_n(t)\mathcal
  F^{-1}\big(\Omega_{N}^{j, n}\hat{f}\big)\bigg\|_{\ell^p}^p{\rm d}t\bigg)^{1/p}\lesssim \log N\|f\|_{\ell^p}.
\end{align}
Observe  that the multiplier from \eqref{eq:24} can be
rewritten as follows
\begin{align*}
 \sum_{M_1\le n\le M_2}\varepsilon_n(t) \Omega_N^{j, n}(\xi)=\sum_{a/q\in\mathscr
    U_{N}}\sum_{M_1\le n\le M_2}\mathfrak
  m_n(\xi-a/q)\eta_N(\xi-a/q)
\end{align*}
with the functions 
\[
\mathfrak m_n(\xi)=\varepsilon_n(t)\Phi\big(2^{nA+jI}\xi\big).
\]
We observe that
\[
|\mathfrak m_n(\xi)|\lesssim\min\big\{|2^{nA+jI}\xi|_{\infty}, |2^{nA+jI}\xi|_{\infty}^{-1}\big\}. 
\] 
The first bound follows from the mean-value theorem, since
\[
\big|\Phi\big(2^{nA+jI}\xi\big)\big|=
\big|\Phi\big(2^{nA+jI}\xi\big)-
\Phi(0)\big|\lesssim\big|2^{nA+jI}\xi\big|\sup_{\xi\in\RR^d}\big|\nabla
\Phi(\xi)\big| \lesssim |2^{nA+jI}\xi|_{\infty}.
\]
The second bound follows since $\Phi$ is a Schwartz
function. Moreover, for every $p\in(1, \infty)$ there is $C_p>0$ such
that 
\[
\big\|\sup_{n\in\ZZ}\big|\calF^{-1}\big(\mathfrak m_n\calF
f\big)\big|\big\|_{L^p}\le C_p \|f\|_{L^p}
\]
for every $f\in L^p\big(\RR^d\big)$. Therefore, by \cite{bigs} the
 multiplier  
\[
\sum_{M_1\le n\le M_2}\mathfrak m_n(\xi)
\]
corresponds to a continuous singular integral, thus
defines a bounded operator 
on $L^p\big(\RR^d\big)$ for all
$p\in(1, \infty)$ with the bound independent of $j\in\ZZ$ and $-\infty\le M_1\le
M_2\le \infty$. Hence, Theorem \ref{th:3} applies and  
the multiplier 
\[
\sum_{M_1\le n\le M_2}\varepsilon_n(t) \Omega_N^{j, n}(\xi)
\]
 defines a
bounded  operator on $\ell^p\big(\ZZ^d\big)$ with the $\log N$ loss, and \eqref{eq:24} is
established.
\end{proof}

\begin{remark}
  If the function $\Phi$ is a real-valued function then we have
  \begin{align}
    \label{eq:52}
   \Big\|\sum_{M_1\le n\le M_2}\mathcal
  F^{-1}\big(\Omega_{N}^{j, n}\hat{f_n}\big)\Big\|_{\ell^p}\le C_p \log N\bigg\|\Big(\sum_{M_1\le n\le M_2}|f_n|^2\Big)^{1/2}\bigg\|_{\ell^p}.    
  \end{align}
This is the dual version  
of inequality \eqref{eq:23} for any sequence of functions $\big(f_n:
M_1\le n\le M_2\big)$ such that 
\[
\bigg\|\Big(\sum_{M_1\le n\le M_2}|f_n|^2\Big)^{1/2}\bigg\|_{\ell^p}<\infty.
\]
\end{remark}

We have  gathered all necessary ingredients to prove inequality
\eqref{eq:84}.

\begin{proof}[Proof of inequality \eqref{eq:84}]

  Let $\chi>0$ and $l\in\NN$ be the numbers whose precise valued will
  be adjusted later. As in \cite{mst1} we will consider for every
  $n\in\NN_0$ the multipliers
\begin{align}
  \label{eq:53}
  \Xi_{n}(\xi)=\sum_{a/q\in\mathscr U_{n^l}}\eta\big(2^{n(A-\chi I)}(\xi-a/q)\big)
\end{align}
with $\mathscr U_{N}$ which has been defined in Section \ref{sec:6}. 
Theorem \ref{th:3} yields, for every $p\in(1, \infty)$, that 
\begin{align}
  \label{eq:54}
  \big\|\mathcal
  F^{-1}\big(\Xi_{n}\hat{f}\big)\big\|_{\ell^p}\lesssim \log(n+2)\|f\|_{\ell^p}.
\end{align}
The implicit constant in \eqref{eq:54} depends on $\rho>0$ from
Theorem \ref{th:3}. From now on we will assume that $l\in\NN$ and
$\rho>0$ are related by the equation
\begin{align}
  \label{eq:62}
  10\rho l=1.
\end{align}
Assume  that $f:\ZZ^d\mapsto\CC$ has finite support and $f\ge0$. Observe that
\begin{align}
  \label{eq:30}
   \Big\|\sum_{n\ge0}T_{n}f\Big\|_{\ell^p}
\le 
  \Big\|\sum_{n\ge0}\mathcal
  F^{-1}\big(m_n\Xi_{n}\hat{f}\big)\Big\|_{\ell^p}
+ 
  \Big\|\sum_{n\ge0}\mathcal
  F^{-1}\big(m_n(1-\Xi_{n})\hat{f}\big)\Big\|_{\ell^p}.
\end{align}
Without of loss of generality we may assume that $p\ge 2$, the case
$1<p\le 2$ follows by the duality then. 
\subsection{The estimate of the second norm in \eqref{eq:30}}

It suffices to show that
\begin{align}
  \label{eq:38}
  \big\|\mathcal
  F^{-1}\big(m_n(1-\Xi_{n})\hat{f}\big)\big\|_{\ell^p}\lesssim (n+1)^{-2}\|f\|_{\ell^p}.
\end{align}
For this purpose we   define for every $x\in\ZZ^d$ the Radon averages  
\[
M_Nf(x)=N^{-k}\sum_{y\in\NN_N^k}f\big(x-\calQ(y)\big).
\]
From \cite{mst1} follows that for every $p\in(1, \infty)$ there is a constant
$C_p>0$ such that for every $f\in\ell^p\big(\ZZ^d\big)$ we have
\begin{align}
  \label{eq:55}
  \big\|\sup_{N\in\NN}|M_Nf|\big\|_{\ell^p}\le C_p\|f\|_{\ell^p}.
\end{align}
Then for every $1<p<\infty$ by \eqref{eq:54} and \eqref{eq:55}  we obtain 
\begin{align}
  \label{eq:21}
  \big\|\mathcal
  F^{-1}\big(m_n(1-\Xi_{n})\hat{f}\big)\big\|_{\ell^p}\le
\big\|\sup_{N\in\NN}M_Nf\big\|_{\ell^p}+\big\|\sup_{N\in\NN}M_N\big(\big|
  \mathcal
  F^{-1}\big(\Xi_{n}\hat{f}\big)\big|\big)\big\|_{\ell^p}\lesssim \log (n+2)\|f\|_{\ell^p}
\end{align}
since we have a pointwise bound
\begin{align}
  \label{eq:63}
\big|\mathcal F^{-1}\big(m_n\hat{f}\big)(x)\big|=|T_nf(x)|\lesssim M_{2^n}f(x).  
\end{align}
We show that it is possible to improve  estimate
\eqref{eq:21} 
for $p=2$. 
Indeed, by Theorem \ref{thm:3} we will show that for big enough $\alpha>0$, which will be specified
later, and for all $n\in\NN_0$  we have
\begin{align}
  \label{eq:60}
 \big|
 m_{n}(\xi)(1-\Xi_{n}(\xi))\big|\lesssim (n+1)^{-\alpha}.
\end{align} 
By Dirichlet's principle we have for every $\gamma\in\Gamma$ 
\[
|\xi_{\gamma}-a_{\gamma}/q_{\gamma}|\le q_{\gamma}^{-1}n^{\beta}2^{-n|\gamma|}
\] 
where $1\le q_{\gamma}\le n^{-\beta}2^{n|\gamma|}$. In order to apply
Theorem \ref{thm:3} we must show that there exists some
$\gamma\in\Gamma$ such that $n^{\beta}\le q_{\gamma }\le
n^{-\beta}2^{n|\gamma|}$. Suppose for a contradiction that for every
$\gamma \in \Gamma$  we have $1\le q_{\gamma }<n^{\beta} $ then for
some 
$q\le \mathrm{lcm}(q_{\gamma}: \gamma\in\Gamma)\le n^{\beta d}$ we have 
\[
|\xi_{\gamma}-a_{\gamma}'/q|\le n^{\beta}2^{-n|\gamma|}
\] 
where $\mathrm{gcd}\big(q, \mathrm{gcd}({a_{\gamma}'}:
\gamma\in\Gamma)\big)=1$.
Hence, taking $a'=(a_{\gamma}': \gamma\in\Gamma)$ we have
$a'/q\in\mathscr U_{n^l}$ provided that $\beta d<l$. On the
other hand, if $1-\Xi_{n}(\xi)\not=0$ then for every
$a'/q\in\mathscr U_{n^l}$ there exists $\gamma\in\Gamma$ such that
\[
|\xi_{\gamma}-a_{\gamma}'/q|>(16d)^{-1} 2^{-n(|\gamma|-\chi)}.
\] 
Therefore
\[
2^{\chi n}<16dn^{\beta}
\] 
 but this is impossible when $n\in\NN$ is large.
Hence, there is
$\gamma\in\Gamma$ such that $n^{\beta}\le q_{\gamma }\le
n^{-\beta}2^{n|\gamma|}$. Thus by Theorem \ref{thm:3}  
\[
|m_n(\xi)|\lesssim (n+1)^{-\alpha}
\]
provided that $1-\Xi_{n}(\xi)\not=0$. This yields \eqref{eq:60} and we
obtain
\begin{align}
\label{eq:56}  
\big\|\mathcal
  F^{-1}\big(m_n(1-\Xi_{n})\hat{f}\big)\big\|_{\ell^2}\lesssim(1+n)^{-\alpha} \log (n+2)\|f\|_{\ell^2}.
\end{align}
Interpolating \eqref{eq:56} with \eqref{eq:21} we obtain 
\begin{align}
\label{eq:57}  
 \big\|\mathcal F^{-1}\big(m_n(1-\Xi_{n})\hat{f}\big)\big\|_{\ell^p}\lesssim(1+n)^{-c_p\alpha} \log (n+2)\|f\|_{\ell^p}.
\end{align}
for some $c_p>0$. Choosing $\alpha>0$ and $l\in\NN$ appropriately
large one obtains \eqref{eq:38}. 

\subsection{The estimate of the first  norm in \eqref{eq:30}}
 Note that for
any $\xi\in\TT^d$ so that
\[
|\xi_{\gamma}-a_{\gamma}/q|\le 2^{-n(|\gamma|-\chi)}
\] 
for every $\gamma\in \Gamma$
with $1\le q\le e^{n^{1/10}}$ we have 
\begin{align}
\label{eq:61}
  m_n(\xi)=G(a/q)\Phi_n(\xi-a/q)+q^{-\delta}E_{2^n}(\xi)
\end{align}
where 
\begin{align}
\label{eq:64}
  |E_{2^n}(\xi)|\lesssim 2^{-n/2}.
\end{align}
 Proposition \ref{prop:0}, with
$L_1=2^n$, $L_2=2^{\chi n}$ and $L_3=e^{n^{1/10}}$, establishes
\eqref{eq:61} and \eqref{eq:64}, since for sufficiently large
$n\in\NN$ we have
\[
q^{\delta}|E_{2^n}(\xi)|\lesssim q^{\delta}L_2L_32^{-n}\lesssim
\big(e^{-n((1-\chi)\log 2-2n^{-9/10})}\big)\lesssim 2^{-n/2}
\]
provided $\chi>0$ is
sufficiently small. Now for
every $j, n\in\NN_0$ we introduce the
 multipliers
\begin{align}
  \label{eq:69}
  \Xi_{n}^j(\xi)=\sum_{a/q\in\mathscr U_{n^l}}\eta\big(2^{nA+jI}(\xi-a/q)\big)^2
\end{align}
and we note that
\begin{align}
  \label{eq:65}
    \Big\|\sum_{n\ge0}\mathcal
  F^{-1}\big(m_n\Xi_{n}\hat{f}\big)\Big\|_{\ell^p}
&\le \Big\|\sum_{n\ge0}\mathcal
  F^{-1}\Big(\sum_{-\chi n\le j<n}m_n\big(\Xi_{n}^j-\Xi_{n}^{j+1}\big)\hat{f}\Big)\Big\|_{\ell^p}\\
\nonumber&+\Big\|\sum_{n\ge0}\mathcal 
  F^{-1}\big(m_n\Xi_{n}^{n}\hat{f}\big)\Big\|_{\ell^p}=I_p^1+I_p^2.
\end{align}
We will estimate $I_p^1$ and $I_p^2$ separately. For this purpose
observe that by \eqref{eq:61} and \eqref{eq:64} for every
$a/q\in\mathscr U_{n^l}$ we have
\begin{align}
\label{eq:66}  
|m_n(\xi)|&\lesssim 
q^{-\delta}|\Phi_{n}(\xi-a/q)|+q^{-\delta}|E_{2^n}(\xi)|\\
  &\lesssim 
\nonumber q^{-\delta}\big(\min\big\{1, |2^{nA}(\xi-a/q)|_{\infty}, |2^{nA}(\xi-a/q)|_{\infty}^{-1/d}\big\}+2^{-n/2}\big)
\end{align}
where the last inequality follows from \eqref{eq:18} and
\eqref{eq:19}. Therefore by \eqref{eq:66} we get
\begin{align}
\label{eq:71}  
 \big|m_{n}(\xi)
\big(\eta\big(2^{nA+jI}(\xi-a/q)\big)^2-\eta\big(2^{nA+(j+1)I}(\xi-a/q)\big)^2\big)\big|
\lesssim q^{-\delta}\big(2^{-|j|/d}+2^{-n/2}\big).
\end{align}

\subsubsection{Bounding $I_p^2$} It will suffice to show, for some
$\varepsilon=\varepsilon_{p}>0$, that
\begin{align}
  \label{eq:102}
  \big\|\mathcal
  F^{-1}\big(m_n\Xi_{n}^{n}\hat{f}\big)\big\|_{\ell^p}\lesssim 2^{-\varepsilon n}\|f\|_{\ell^p}.
\end{align}
Observe that for any $1<p<\infty$ by \eqref{eq:63}, \eqref{eq:55} and \eqref{eq:54} we have
\begin{align}
 \label{eq:103}\big\|\mathcal
  F^{-1}\big(m_n\Xi_{n}^{n}\hat{f}\big)\big\|_{\ell^p}\le\big\|\sup_{N\in\NN}
  M_N\big(\big|\mathcal
  F^{-1}\big(\Xi_{n}^{n}\hat{f}\big)\big|\big)\big\|_{\ell^p}\lesssim
\big\|\mathcal
  F^{-1}\big(\Xi_{n}^{n}\hat{f}\big)\big\|_{\ell^p}
  \lesssim \log (n+2)\|f\|_{\ell^p}.
\end{align}
 For $p=2$ by Plancherel's theorem and by \eqref{eq:66} we obtain
 \begin{multline}
   \label{eq:105}
   \big\|\mathcal
  F^{-1}\big(m_n\Xi_{n}^{n}\hat{f}\big)\big\|_{\ell^2}
=
\bigg(\int_{\TT^d}\sum_{a/q\in\mathscr U_{n^l}}|m_{n}(\xi)|^2\eta\big(2^{nA+nI}(\xi-a/q)\big)^4
  |\hat{f}(\xi)|^2d\xi\bigg)^{1/2}\lesssim 
2^{-n/(2d)}\|f\|_{\ell^2}.
 \end{multline}
Therefore, interpolating \eqref{eq:103} with \eqref{eq:105} we obtain
for every $p\in(1, \infty)$ that 
\begin{align*}
 \big\|\mathcal
  F^{-1}\big(m_n\Xi_{n}^{n}\hat{f}\big)\big\|_{\ell^p}\lesssim 2^{-\varepsilon n}\|f\|_{\ell^p}
\end{align*}
which in turn implies \eqref{eq:102} and $I_p^2\lesssim \|f\|_{\ell^p}$.
\subsubsection{Bounding $I_p^1$}
Define for any $0\le s< n$ new multipliers 
\begin{align*}
  \label{eq:68}
  \Delta_{n, s}^j(\xi)=\sum_{a/q\in\mathscr
    U_{(s+1)^l}\setminus\mathscr
    U_{s^l}}\big(\eta\big(2^{nA+jI}(\xi-a/q)\big)^2-
\eta\big(2^{nA+(j+1)I}(\xi-a/q)\big)^2\big)\eta\big(2^{s(A-\chi I)}(\xi-a/q)\big)^2
\end{align*}
and we observe that by the definition \eqref{eq:69} we have
\begin{align*}
  \Xi_{n}^j(\xi)-\Xi_{n}^{j+1}(\xi)=\sum_{0\le s<n}\Delta_{n, s}^j(\xi).
\end{align*}
Moreover,  
\begin{multline*}
\eta\big(2^{nA+jI}(\xi)\big)^2-
\eta\big(2^{nA+(j+1)I}(\xi)\big)^2\\
=
\big(\eta\big(2^{nA+jI}(\xi)\big)^2-
\eta\big(2^{nA+(j+1)I}(\xi)\big)^2\big)
\cdot\big(\eta\big(2^{nA+(j-1)I}(\xi)\big)-
\eta\big(2^{nA+(j+2)I}(\xi)\big)\big).  
\end{multline*}
Thus we see 
\[
\Delta_{n, s}^j(\xi)=\Delta_{n, s}^{j, 1}(\xi)\cdot\Delta_{n, s}^{j, 2}(\xi),
\]
where
\begin{align*}
  \Delta_{n,s}^{j, 1}(\xi)=\sum_{a/q\in\mathscr
    U_{(s+1)^l}\setminus\mathscr
    U_{s^l}}\big(\eta\big(2^{nA+(j-1)I}(\xi-a/q)\big)-
\eta\big(2^{nA+(j+2)I}(\xi-a/q)\big)\big)\eta\big(2^{s(A-\chi I)}(\xi-a/q)\big)
\end{align*}
and
\begin{align*}
\Delta_{n, s}^{j, 2}(\xi)=\sum_{a/q\in\mathscr
    U_{(s+1)^l}\setminus\mathscr
    U_{s^l}}\big(\eta\big(2^{nA+jI}(\xi-a/q)\big)^2-
\eta\big(2^{nA+(j+1)I}(\xi-a/q)\big)^2\big)\eta\big(2^{s(A-\chi I)}(\xi-a/q)\big).  
\end{align*}
Moreover, $\Delta_{n,s}^{j, 1}$ and $\Delta_{n,s}^{j, 2}$ are the
multipliers which satisfy the assumptions of Theorem
\ref{thm:30}. Therefore, 
\begin{multline}
  \label{eq:70}
  I_p^1=\Big\|\sum_{n\ge0}\mathcal
  F^{-1}\Big(\sum_{-\chi n\le j< n}\sum_{0\le s<n}\Delta_{n, s}^{j,
    1}m_n\Delta_{n, s}^{j, 2}\hat{f}\Big)\Big\|_{\ell^p}\\
\le
\sum_{s\ge0}\sum_{j\in \ZZ}\Big\|\sum_{n\ge\max\{j,-j/\chi, s\}}\mathcal
  F^{-1}\Big(\Delta_{n, s}^{j, 1}m_n\Delta_{n, s}^{j,
    2}\hat{f}\Big)\Big\|_{\ell^p}\\
\lesssim \sum_{s\ge0}\sum_{j\in\ZZ}\log s\bigg\|\Big(\sum_{n\ge\max\{j,-j/\chi, s\}}\big|\mathcal
  F^{-1}\big(m_n\Delta_{n, s}^{j, 2}\hat{f}\big)\big|^2\Big)^{1/2}\bigg\|_{\ell^p}.
\end{multline}
In the last step we have used \eqref{eq:52}.
The task now is to show that for some $\varepsilon=\varepsilon_p>0$ 
\begin{align}
  \label{eq:73}
  \bigg\|\Big(\sum_{n\ge\max\{j,-j/\chi, s\}}\big|\mathcal
  F^{-1}\big(m_n\Delta_{n, s}^{j,
    2}\hat{f}\big)\big|^2\Big)^{1/2}\bigg\|_{\ell^p}
\lesssim
  s^{-2}2^{-\varepsilon j}\|f\|_{\ell^p}.
\end{align}
This in turn will imply $I_p^1\lesssim \|f\|_{\ell^p}$ and the proof
will be completed. We have assumed that $p\ge 2$, then for
every $g\in\ell^r(\ZZ^d)$ such that $g\ge 0$ with $r=(p/2)'>1$ we have
by  \eqref{eq:63}, the Cauchy--Schwarz inequality and \eqref{eq:55} that 
\begin{multline}
  \label{eq:40}
  \sum_{x\in\ZZ^d}\sum_{n\in\ZZ}\big|\mathcal
  F^{-1}\big(m_n\Delta_{n, s}^{j,
    2}\hat{f}\big)(x)\big|^2g(x)
\lesssim \sum_{x\in\ZZ^d}\sum_{n\in\ZZ}M_{2^n}\big(\big|\mathcal
  F^{-1}\big(\Delta_{n, s}^{j,
    2}\hat{f}\big)\big|\big)(x)^2g(x)\\
\le \sum_{x\in\ZZ^d}\sum_{n\in\ZZ}M_{2^n}\big(\big|\mathcal
  F^{-1}\big(\Delta_{n, s}^{j,
    2}\hat{f}\big)\big|^2\big)(x)g(x)=
\sum_{x\in\ZZ^d}\sum_{n\in\ZZ}\big|\mathcal
  F^{-1}\big(\Delta_{n, s}^{j,
    2}\hat{f}\big)(x)\big|^2 M_{2^n}^*g(x)\\
\lesssim
\bigg\|\Big(\sum_{n\in\ZZ}\big|\mathcal
  F^{-1}\big(\Delta_{n, s}^{j,
    2}\hat{f}\big)\big|^2\Big)^{1/2}\bigg\|_{\ell^p}^2
 \big\|\sup_{N\in\NN} M_N^*g\big\|_{\ell^r}
\lesssim
\bigg\|\Big(\sum_{n\in\ZZ}\big|\mathcal
  F^{-1}\big(\Delta_{n, s}^{j,
    2}\hat{f}\big)\big|^2\Big)^{1/2}\bigg\|_{\ell^{p}}^2
 \|g\|_{\ell^r}.
\end{multline}
Therefore, by Theorem \ref{thm:30} we have
\begin{align}
  \label{eq:43}
\bigg\|\Big(\sum_{n\in\ZZ}\big|\mathcal
  F^{-1}\big(m_n\Delta_{n, s}^{j,
    2}\hat{f}\big)\big|^2\Big)^{1/2}\bigg\|_{\ell^p}
\lesssim
\bigg\|\Big(\sum_{n\in\ZZ}\big|\mathcal
  F^{-1}\big(\Delta_{n, s}^{j,
    2}\hat{f}\big)\big|^2\Big)^{1/2}\bigg\|_{\ell^p}\lesssim \log s\|f\|_{\ell^p}.  
\end{align}

We refine the estimate in \eqref{eq:43} for
$p=2$. Indeed, define 
\[
\varrho_{n, j}(\xi)=\big(\eta\big(2^{nA+jI}(\xi)\big)^2-
\eta\big(2^{nA+(j+1)I}(\xi)\big)^2\big)
\eta\big(2^{s(A-\chi I)}(\xi)\big)
\] 
and 
\[
\Psi_n(\xi)=\min\big\{|2^{nA}\xi|_{\infty}, |2^{nA}\xi|_{\infty}^{-1/d}, 1\big\}.
\]
By Plancherel's theorem we have
\begin{multline}
  \label{eq:46}
\bigg\|\Big(\sum_{n\ge\max\{j,-j/\chi, s\}}\big|\mathcal
  F^{-1}\big(m_n\Delta_{n, s}^{j,
    2}\hat{f}\big)\big|^2\Big)^{1/2}\bigg\|_{\ell^2}\\
  =\bigg(\int_{\TT^d}\sum_{n\ge\max\{j,-j/\chi, s\}}\sum_{a/q\in\mathscr
    U_{(s+1)^l}\setminus\mathscr
   U_{s^l}}|m_n(\xi)|^2\varrho_{n,j}(\xi-a/q)^2|\hat f(\xi)|^2{\rm
   d}\xi\bigg)^{1/2}\\
\lesssim (s+1)^{-\delta l}2^{-|j|/(2d)}\|f\|_{\ell^2}.
\end{multline}
The last estimate is implied by \eqref{eq:66}. Namely, by
\eqref{eq:66} we may write
\begin{multline}
  \label{eq:72}
\sum_{n\ge\max\{j,-j/\chi, s\}}\sum_{a/q\in\mathscr
    U_{(s+1)^l}\setminus\mathscr
   U_{s^l}}|m_n(\xi)|^2\varrho_{n,j}(\xi-a/q)^2\\
\lesssim
\sum_{n\ge\max\{j,-j/\chi, s\}}\sum_{a/q\in\mathscr
    U_{(s+1)^l}\setminus\mathscr
    U_{s^l}}q^{-2\delta}\big(\Psi_n(\xi-a/q)+2^{-n/2}\big)\big(2^{-|j|/d}+2^{-n/2}\big)\eta\big(2^{s(A-\chi
  I)}(\xi-a/q)\big)^2\\
\lesssim (s+1)^{-2\delta l}2^{-|j|/(2d)}.
\end{multline}
The last line follows, since we have used the lower bound for $q\ge s^{l}$ if
$a/q\in\mathscr U_{(s+1)^l}\setminus\mathscr U_{s^l}$. Moreover,
\[
\sum_{n\ge0}\big(\Psi_n(\xi-a/q)+2^{-n/2}\big)\lesssim 1
\]
and 
\[
\sum_{a/q\in\mathscr
    U_{(s+1)^l}\setminus\mathscr
    U_{s^l}}\eta\big(2^{s(A-\chi I)}(\xi-a/q)\big)\lesssim 1
\]
by the disjointness of the supports of $\eta\big(2^{s(A-\chi
  I)}(\xi-a/q)\big)$'s whenever $a/q\in\mathscr
    U_{(s+1)^l}\setminus\mathscr
    U_{s^l}$. Since $l\in\NN$ can be as large as we
wish thus interpolating  \eqref{eq:46} with \eqref{eq:43} we obtain
\eqref{eq:73} and the proof of \eqref{eq:84} and consequently Theorem \ref{thm:0} is completed. 

\end{proof}

\begin{bibliography}{discrete}
	\bibliographystyle{amsplain}

\providecommand{\bysame}{\leavevmode\hbox to3em{\hrulefill}\thinspace}
\providecommand{\MR}{\relax\ifhmode\unskip\space\fi MR }
\providecommand{\MRhref}[2]{%
  \href{http://www.ams.org/mathscinet-getitem?mr=#1}{#2}
}
\providecommand{\href}[2]{#2}
\begin{thebibliography}{10}

\bibitem{bou1}
J.~Bourgain, \emph{{O}n the maximal ergodic theorem for certain subsets of the
  integers}, Israel J. Math. \textbf{61} (1988), 39--72.

\bibitem{bou2}
\bysame, \emph{{O}n the pointwise ergodic theorem on {$L^p$} for arithmetic
  sets}, Israel J. Math. \textbf{61} (1988), 73--84.

\bibitem{bou}
\bysame, \emph{Pointwise ergodic theorems for arithmetic sets. {W}ith an
  appendix by the author, {H}arry {F}urstenberg, {Y}itzhak {K}atznelson and
  {D}onald {S}. {O}rnstein.}, Publ. Math.-Paris \textbf{69} (1989), no.~1,
  5--45.

\bibitem{cnsw}
A.~Stein~E.M. Christ, M.~Nagel and S.~Wainnger, \emph{Singular and maximal
  radon transforms: analysis and geometry}, Ann. of Math. \textbf{150} (1999),
  489--577.

\bibitem{iw}
A.D. Ionescu and S.~Wainger, \emph{{$L^p$} boundedness of discrete singular
  {R}adon transforms}, J. Amer. Math. Soc. \textbf{19} (2006), no.~2, 357--383.

\bibitem{mst1}
M.~Mirek, E.M. Stein, and B.~Trojan, \emph{{$\ell^p\big(\ZZ^d\big)$}-estimates
  for discrete operators of {R}adon type: {M}aximal functions and vector-valued
  estimates}, Preprint, 2015.

\bibitem{mst2}
\bysame, \emph{{$\ell^p\big(\ZZ^d\big)$}-estimates for discrete operators of
  {R}adon type: {V}ariational estimates}, Preprint, 2015.

\bibitem{bigs}
E.M. Stein, \emph{{H}armonic {A}nalysis: {R}eal-{V}ariable {M}ethods,
  {O}rthogonality, and {O}scillatory {I}ntegrals}, Princeton University Press,
  Princeton, 1993.

\bibitem{SW0}
E.M. Stein and S.~Wainger, \emph{Discrete analogues in harmonic analysis, {I}:
  {$\ell^2$} estimates for singular {R}adon transforms}, Amer. J. Math.
  \textbf{121} (1999), no.~6, 1291--1336.

\bibitem{sw}
\bysame, \emph{Oscillatory integrals related to {C}arleson's theorem}, Math.
  Res. Lett. \textbf{8} (2001), 789--800.

\end{thebibliography}
\end{bibliography}

\end{document}